\documentclass[12pt]{amsart}
\usepackage{amsmath,amsthm,mathrsfs,amsfonts,amssymb,color}
\newtheorem{thm}{Theorem}[section]
\newtheorem{lem}[thm]{Lemma}

\newtheorem{pro}[thm]{Proposition}
\newtheorem*{defn}{Definition}
\newtheorem{rem}{\it Remark}[section]

\numberwithin{equation}{section}
\allowdisplaybreaks
\begin{document}

\title[]%[Monge-Amp\`ere SIO's on Besov spaces]
{Boundedness of Monge-Amp\`ere singular integral operators on Besov spaces}

\author[]{Yongshen Han, Ming-Yi Lee and Chin-Cheng Lin}

\thanks{The second and third authors are supported by Ministry of Science and Technology, R.O.C. under
Grant \#MOST 106-2115-M-008-003-MY2 and Grant \#MOST 106-2115-M-008-004-MY3, respectively,
as well as supported by National Center for Theoretical Sciences of Taiwan.}


\begin{abstract}
Let $\phi: \Bbb R^n \mapsto \Bbb R$ be a strictly convex and smooth function, and
$\mu= \text{det}\,D^2 \phi$ be the Monge-Amp\`ere measure generated by $\phi.$
For $x\in \Bbb R^n$ and $t>0$, let $S(x,t):=\{y\in \Bbb R^n: \phi(y)<\phi(x)+\nabla \phi(x)\cdot(y-x)+t\}$ denote the section.
If $\mu$ satisfies the doubling property, Caffarelli and Guti\'errez (Trans. AMS 348:1075--1092, 1996) provided a variant of the Calder\'on-Zygmund decomposition and a John-Nirenberg-type inequality associated with sections. Under a stronger uniform continuity condition on $\mu$, they also (Amer. J. Math. 119:423--465, 1997) proved an invariant Harnack's inequality for nonnegative solutions of the Monge-Amp\`ere equations with respect to sections. The purpose of this paper is to establish a theory of Besov spaces associated with sections under only the doubling condition on $\mu$ and prove that Monge-Amp\`ere singular integral operators are bounded on these spaces.
\end{abstract}

\maketitle
\section {Introduction}\label{sec-Intro}

Let $\phi: \Bbb R^n \mapsto \Bbb R$ be a strictly convex and smooth function and
consider the Monge-Amp\`ere measure $\mu$ generated by $\phi$
$$\mu:= \text{det}\,D^2 \phi,$$
where $D^2 \phi$ denotes the Hessian matrix of $\phi$. For a given function $u$,
$$\text{det}\,D^2 (\phi+tu)=\text{det}\,D^2 \phi+t\,\text{trace}(\Phi\,D^2 u)+\ldots +t^n\text{det}\,D^2 u,$$
where $\Phi=(\Phi)_{ij}$ is the matrix of cofactors of $D^2 \phi$.
The linearization of the Monge-Amp\'ere equation is denoted by
$$L_\phi u = \text{trace}(\Phi\,D^2 u).$$
To study the properties of the solutions for the equation $L_\phi u =0,$ Caffarelli and Guti\'errez \cite{CG1} introduced a family of {\it sections} as follows.
Let $\rho(x,y)=\phi(y)-\phi(x)-\nabla \phi(x)\cdot(y-x)$.
Given $x\in \Bbb R^n$ and $t>0$, the section is defined by
$$S(x,t)=S_{\phi}(x,t)=\{y\in \Bbb R^n: \rho(x,y)<t\}.$$
These sets are convex and play crucial role in the study of Monge-Amp\`ere equation and the linearized
Monge-Amp\`ere equation (see \cite{C1, C2, CG1, CG2}). Indeed, if the Monge-Amp\`ere measure $\mu$ satisfies the geometric conditions, namely doubling and a uniform continuity conditions,
Caffarelli and Guti\'errez \cite{CG1, CG2} proved a variant of the Calder\'on-Zygmund decomposition and a John-Nirenberg-type inequality associated with sections and an invariant Harnack's inequality with respect to sections. To be more precise, it was assumed in \cite{CG1} that the Monge-Amp\`ere measure $\mu$ satisfies the following property: there exist constants $C>0$ and $0<\alpha<1$ such that
\begin{equation*}
\mu(S(x,t))\le C\mu(\alpha S(x,t))\qquad\text{for all}\ S(x,t),
\end{equation*}
where $\alpha S(x,t)$ denotes the $\alpha$-dilation of the section $S(x,t)$ with respect to its center of mass.
It was proved in \cite{C1} that sections satisfying this hypothesis on $\mu$ imply that the graph of $\phi$ does not contain segments of lines
and the sections $S(x,t)$ are of a size that can be controlled by Euclidean balls when these sections are rescaled by using appropriate affine transformations.
Under these conditions, Caffarelli and Guti\'errez \cite{CG1} proved a variant of the Calder\'on-Zygmund decomposition and a John-Nirenberg-type inequality associated with sections.
However, to obtain an invariant Harnack's inequality on the sections, it requires a stronger uniform continuity condition on $\mu$,
namely, for any given $\delta_1\in  (0, 1)$, there exists $\delta_2\in  (0, 1)$ such that, for all sections $S$ and all measurable subset $E\subset S,$
if $|E|<\delta_2 |S|$, then $\mu(E)<\delta_1 \mu(S).$ Under this uniform continuity condition on $\mu,$
Caffarelli and Guti\'errez \cite{CG2} showed an invariant Harnack's inequality on sections as follows.

\begin{thm}\label{Har}
There exist constants $\beta>1$ and $0<\tau< \frac 13$ depending only on the structure such that if $u$ is any nonnegative solution of $L_\phi u=0$ in the section $S(z,t)$, then
$$\sup_{S(z, \tau t)}u\leq \beta \inf_{S(z, \tau t)}u.$$
\end{thm}

As pointed in \cite{CG1}, sections satisfy the following conditions:
\begin{enumerate}
\item[({\bf A})] There exist positive constants $K_1, K_2,
      K_3$ and $\epsilon_1, \epsilon_2$ such that given two sections
      $S(x_0,t_0)$, $S(x,t)$ with $t\le t_0$ satisfying
$$S(x_0,t_0)\cap S(x,t)\neq \varnothing,$$
      and an affine transformation $T$ that ``normalizes" $S(x_0,t_0)$; that is,
$$B(0,1/n)\subset T(S(x_0,t_0))\subset B(0,1),$$
      there exists $z\in B(0, K_3)$ depending on $S(x_0,t_0)$ and $S(x,t)$, which satisfies
$$B\big(z,K_2(t/t_0)^{\epsilon_2}\big)\subset T(S(x,t))\subset
      B\big(z,K_1(t/t_0)^{\epsilon_1}\big),$$
      and
$$T(x)\in B\big(z, (1/2)K_2 (t/t_0)^{\epsilon_2}\big).$$
     Here and below $B(x,t)$ denotes the Euclidean ball centered at
     $x$ with radius $t$.
\item[({\bf B})] There exists a constant $\nu>0$ such that given a section
     $S(x,t)$ and $y\notin S(x,t)$, if $T$ is an affine transformation that
     normalizes $S(x,t)$, then, for any $0<\epsilon<1$,
     $$B(T(y),\epsilon^\nu)\cap T(S(x,(1-\epsilon)t))=\varnothing.$$
\item[({\bf C})] $\bigcap_{t>0}S(x,t)=\{x\}$ and $\bigcup_{t>0}S(x,t)=\Bbb R^n.$
\end{enumerate}
Based on the above properties on sections, Caffarelli and Guti\'errez \cite{CG3} introduced the Monge-Amp\`ere singular integral operators as follows.
Suppose that $0<\gamma\le 1$ and $c_1, c_2>0$.
Let $\{k_i(x,y)\}_{i\in \Bbb Z}$ be a sequence of kernels satisfying
the following conditions:
\begin{enumerate}
\item[({\bf D}1)]  {\rm supp}\,$k_i(\cdot,y)\subset S(y,2^i)\ $ for all $y\in \Bbb R^n$;
\item[({\bf D}2)]  {\rm supp}\,$k_i(x,\cdot)\subset S(x,2^i)\ $ for all $x\in \Bbb R^n$;
\item[({\bf D}3)]  $\displaystyle\int_{\Bbb R^n}k_i(x,y)d\mu(y)=
\int_{\Bbb R^n}k_i(x,y)d\mu(x)=0\ $ for all $x, y\in \Bbb R^n$;
\item[({\bf D}4)]  $\displaystyle\sup_i\int_{\Bbb R^n}|k_i(x,y)|d\mu(y)\le c_1\ $
for all $x\in \Bbb R^n$;
\item[({\bf D}5)]  $\displaystyle\sup_i\int_{\Bbb R^n}|k_i(x,y)|d\mu(x)\le c_1\ $
for all $y\in \Bbb R^n$;
\item[({\bf D}6)]  If $T$ is an affine transformation that normalizes
the section $S(y,2^i)$, then
$$|k_i(u,y)-k_i(v,y)|\le \frac {c_2}{\mu(S(y,2^i))}|T(u)-T(v)|^\gamma;$$
\item[({\bf D}7)]  If $T$ is an affine transformation that normalizes
the section $S(x,2^i)$, then
$$|k_i(x,u)-k_i(x,v)|\le \frac {c_2}{\mu(S(x,2^i))}|T(u)-T(v)|^\gamma.$$
\end{enumerate}
Denote $K(x,y)=\sum_{i\in \Bbb Z} k_i(x,y)$.
The {\it Monge-Amp\`ere singular integral operator} $H$ is defined by
$$H(f)(x)=\int_{\Bbb R^n} K(x,y)f(y)d\mu(y).$$
Caffarelli and Guti\'errez \cite{CG3} proved that $H$
is bounded on $L^2(\Bbb R^n, d\mu)$. Subsequently, Incognito \cite{In} established the
$L^p(\Bbb R^n, d\mu)$, $1<p<\infty$, and weak type (1,1) estimates of $H$.
It was also showed that $H$ is bounded from $H^1_{\mathcal F}(\Bbb R^n)$ to $L^1(\Bbb R^n, d\mu)$ and is bounded on $H^1_{\mathcal F}(\Bbb R^n)$ in \cite{DL} and \cite{Le}, respectively. Recently, Lin \cite{Li} proved the boundedness of $H$ acting on $H^p_{\mathcal F}(\Bbb R^n)$, $1/2 <p \le 1$, and their dual spaces which can be realized as Carleson measure spaces, Campanato spaces, and Lipschitz spaces.

The purpose of this paper is to establish a theory of Besov spaces associated with sections under only the doubling condition on $\mu$ and prove that Monge-Amp\`ere singular integral operators are bounded on these spaces.

It is known that conditions ({\bf A}) and ({\bf B}) imply the following {\it engulfing property}: there exists a constant $\theta\ge1$,
depending only on $\nu, K_1$, and $\epsilon_1$, such that
if $x\in S(y,t)$ then $S(y,t)\subset S(x,\theta t)$.
From this property it is easy to show that
$$\rho(y,x)\le \theta\rho(x,y)$$
and
$$\rho(x,y)\le \theta^2 \big(\rho(x,z)+\rho(z,y)\big).$$
Let $\bar \rho(x,y):=\frac 12(\rho(x,y)+\rho(y,x))$. Then $\bar \rho$ is a quasi-metric on $\Bbb R^n$ in the sense of Coifman and Weiss; that is,
\begin{enumerate}
\item[(i)] $\bar \rho(x,y)= \bar \rho(y,x)\geq 0$ for all $x$, $y\in\Bbb R^n$;
\item[(ii)] $\bar \rho(x,y) = 0$ if and only if $x = y$;
\item[(iii)] the \emph{quasi-triangle inequality} holds: there is a constant $A_0\in [1,\infty)$ such that
\begin{eqnarray}\label{eqn:quasitriangleineq}
    \bar \rho(x,y) \le A_0 [\bar \rho(x,z) + \bar \rho(z,y)]\qquad\text{for all}\ x, y, z\in\Bbb R^n.
\end{eqnarray}
\end{enumerate}
Moreover, it is easy to see that $\bar \rho(x,y)$ and $\rho(x,y)$ are geometrically equivalent due to the fact that $\frac 12{\rho(x,y)}\le \bar\rho(x,y)\le \frac{(1+\theta)}2 \rho(x,y)$.
Therefore, all results obtained by Caffarelli and Guti\'errez as mentioned above still hold with replacing $\rho(x,y)$ by $\bar\rho(x,y)$.
From now on, for simplicity we still use the same notation $S(x,t):=\{y\in \Bbb R^n : \bar\rho(x,y)<t\}$ to denote the sections induced by $\bar\rho$
and let $\mathcal F=\{S(x,t) : x\in \Bbb R^n \ \text{and}\ t>0\}$ be the family of sections.
Since $(\Bbb R^n, \bar\rho, \mu)$ is a space of homogeneous type in the sense of Coifman and Weiss,
one might expect that the Besov space and the boundedness of singular integrals associated with sections deduced
by $\bar\rho$ would follow from known results on spaces of homogeneous type. However, this is {\bf not} the case.
To see this, let us recall the theory of classical Besov spaces on $\Bbb R^n.$
It was well known that the Littlewood-Paley theory plays a crucial role for developing function spaces on $\Bbb R^n.$ Let $\psi$ be a Schwartz function satisfying
\begin{enumerate}
\item[(i)] {\rm supp}\,${\hat\psi }\subset\{\xi\in \Bbb R^n: \frac{1}{2}\leq |\xi|\leq 2\};$
\item[(ii)] $|{\hat\psi}(\xi)|\ge C>0$ for $\{\frac{3}{5}\le |\xi|\le\frac{5}{3}\}.$
\end{enumerate}
The classical Besov space $\dot{B}^{\alpha,q}_p(\Bbb R^n)$ is the set of all
$f\in {\mathscr S}'/\mathscr P(\Bbb R^n)$, the space of tempered distributions modulo polynomials, satisfying
$$\|f\|_{\dot{B}^{\alpha,q}_p}:= \bigg(\sum_{k\in \Bbb Z}\Big(2^{k\alpha}\|{\psi_k}*f\|_p\Big)^q\bigg)^{1/q}<\infty,$$
where $\psi_k(x)=2^{kn}\psi(2^kx)$ for $x\in \Bbb R^n$ and $k\in \Bbb Z$.

A crucial tool for the study of Besov spaces is the Calder\'on reproducing formula which was first provided by Calder\'on \cite{Ca}.
This formula says that, for any given function $\psi$ satisfying the above conditions (i) and (ii), there exists a function $\phi$ with the properties similar to $\psi$ such that
\begin{eqnarray}\label{cal1}
f=\sum_{k=-\infty}^{\infty}\phi_k *\psi_k *f,
\end{eqnarray}
where the series converges not only in $L^2(\Bbb R^n)$, but also in ${\mathscr S}_\infty(\Bbb R^n) =\{f\in {\mathscr S}(\Bbb R^n): \int_{\Bbb R^n} f(x) x^\alpha dx =0$
for all $|\alpha|\ge 0\}$ and in ${\mathscr S}'(\Bbb R^n)$, the dual of ${\mathscr S}_\infty(\Bbb R^n)$. See \cite{FJW} for more details.

Applying this reproducing formula, one can show that the definition of $\dot{B}^{\alpha,q}_p(\Bbb R^n)$ is independent of the choice of functions $\psi$
which satisfy the above conditions (i) and (ii). Moreover, using this formula,
one also can study the theory of the Besov spaces which includes the embedding, interpolation, duality, atomic decomposition, and the boundedness
of singular integrals on $\dot{B}^{\alpha,q}_p(\Bbb R^n)$. See \cite{FJW, P, T1, T2} for more details.

The classical theory of Calder\'on-Zygmund singular integral operators as well as the theory of
function spaces on $\Bbb R^n$ were based on extensive use of convolution operators and on the Fourier transform.
However, it is now possible to extend most of those ideas and results to spaces of homogeneous type.
Spaces of homogeneous type were introduced by
Coifman and Weiss \cite{CW1} in the early 1970's. We say
that $(X,d,\mu)$ is a {\it space of homogeneous type} in the
sense of Coifman and Weiss if $d$ is a quasi-metric on $X$
and $\mu$ is a nonzero measure satisfying the doubling
condition. A \emph{quasi-metric} $d$ on a set $X$ is a
function $d:X\times X\mapsto [0,\infty)$ satisfying
\begin{enumerate}
\item[(i)] $d(x,y) = d(y,x) \geq 0$ for all $x,y\in X$;
\item[(ii)] $d(x,y) = 0$ if and only if $x = y$;
\item[(iii)] the \emph{quasi-triangle inequality}: there is a constant $C\in [1,\infty)$ such that
$$d(x,y) \le C [d(x,z) + d(z,y)]\qquad\text{for all}\ x, y, z\in X.$$
\end{enumerate}
We say that a nonzero measure $\mu$ satisfies the
\emph{doubling condition} if there is a constant $C_\mu$ such
that, for all $x\in X$ and $r > 0$,
\begin{eqnarray}\label{doubling condition}
\mu(B_d(x,2r)) \leq C_\mu \mu(B_d(x,r)) < \infty,
\end{eqnarray}
where $B_d(x,r)=\{y\in X: d(x,y)<r\}$.

Spaces of homogeneous type include many special spaces in
analysis and have many applications in the theory of singular
integrals and function spaces. See \cite{CW1,CW2,NS1,NS2} for more
details.

By the end of the 1970's, it was well recognized that much contemporary real analysis requires little structure on the underlying space. For instance, to obtain the maximal function characterizations for the Hardy spaces on spaces of homogeneous type, Mac\'ias and Segovia \cite{MS1} proved that one can
replace the quasi-metric $d$ by another quasi-metric $d'$
on $X$ such that the topologies induced on $X$ by $d$
and $d'$ coincide, and $d'$ has the following regularity property:
\begin{eqnarray}\label{smooth metric}
    |d'(x,y) - d'(x',y)|
    \le C_0 \, d'(x,x')^\varepsilon \,
        [d'(x,y) + d'(x',y)]^{1 - \varepsilon}
\end{eqnarray}
for some constant $C_0$, some regularity exponent $\varepsilon\in(0,1)$, and for all $x$, $x'$, $y\in X$.
Moreover, the measure $\mu$ satisfies
\begin{eqnarray}\label{regular}
   C_1^{-1} r\le \mu(B_{d'}(x,r))\le C_1 r\qquad\text{for some constant}\ C_1.
\end{eqnarray}
Note that property \eqref{regular} is much stronger than the
doubling condition \eqref{doubling condition}. Mac\'{i}as and Segovia \cite{MS2} established the
maximal function characterization for Hardy spaces $H^p(X)$,
$(1 + \varepsilon)^{-1} < p \le 1$, on spaces of homogeneous
type $(X,d',\mu)$ whenever $d'$ and $\mu$ satisfy the regularity
condition \eqref{smooth metric} and property \eqref{regular}, respectively.

The seminal result on spaces of homogeneous type $(X, d', \mu)$ where $d'$ satisfies the condition \eqref{smooth metric} and $\mu$ satisfies the property \eqref{regular} is the $Tb$ theorem given by David, Journ\'e and Semmes \cite{DJS}.
The key step to establish such a $Tb$ theorem is Coifman's construction of the approximation to the identity and the decomposition of the identity.
Coifman's construction of the approximation to the identity is as  follows.
Take a smooth function $h$ defined on
$[0, \infty)$, equals to 1 on $[1, 2]$, and 0 on $[0, 1/2]\cup [4, \infty)$.
Let ${T}_k$ be the operator with
kernel $2^kh(2^k d'(x,y))$. Property \eqref{regular} of the
measure $\mu$ implies that $C^{-1} \leq {T}_k(1) \leq C$ for some $C> 0$.
Let ${M}_k$ and ${W}_k$ be the operators
of multiplications by $1/{T}_k(1)$ and $\{{T}_k[1/{T}_k(1)]\}^{-1}$,
respectively, and let ${S}_k := { M}_k{T}_k{ W}_k{ T}_k{M}_k$. Then
the regularity property \eqref{smooth metric} on the
metric $d$ and property \eqref{regular} on the measure $\mu$
imply that the kernels $S_k(x,y)$ of ${S}_k$ satisfy the
following conditions: for some constants $C > 0$ and $\varepsilon > 0$,
\begin{enumerate}
\item[(i)] ${S}_k(x,y) = 0$ for $d'(x,y) \ge C2^{-k}$,
           and $\| {S}_k\|_{\infty} \leq C2^k$;
\item[(ii)] $|{ S}_k(x,y)-{ S}_k(x',y)| \le C2^{k(1+\varepsilon)}d'(x,x')^{\varepsilon}$;
\item[(iii)] $|{ S}_k(x,y)-{S}_k(x,y')| \le C2^{k(1+\varepsilon)}d'(y,y')^{\varepsilon}$;
\item[(iv)] $\displaystyle \int_{X}{ S}_k(x,y) \, d\mu(y)= \int_{X}{S}_k(x,y) \, d\mu(x)=1$.
\end{enumerate}

Let ${D}_k := {S}_k - {S}_{k-1}$. Coifman's decomposition of the identity is given as follows. If $\mu(X) = \infty$, the identity operator $I$ can be written as
$$I=\sum_{k=-\infty}^\infty {D}_k=\sum_{k=-\infty}^\infty\sum_{j=-\infty}^\infty {D}_k{D}_j={T}_N +{R}_N,$$
where ${ T}_N=\sum_{|k-j|\leq N} {D}_k{ D}_j$ and ${R}_N=\sum_{|k-j|>N}{D}_k{D}_j.$

David, Journ\'e and Semmes showed that if $N$ is a fixed large integer, then ${R}_N$ is bounded on $L^p(X)$, $1<p<\infty,$ with the operator norm less than 1.
Therefore, if $N$ is a fixed large integer and ${D}_k^N=\sum_{|j|\leq N}{D}_{j+k},$ they obtained the following Calder\'on-type reproducing formula
\begin{eqnarray*}\label{cal2}
f=\sum_{k=-\infty}^\infty {T}_N^{-1}{D}_k^N{D}_k(f)=\sum_{k=-\infty}^\infty {D}_k^N{D}_k{T}_N^{-1}(f),
\end{eqnarray*}
where ${T}_N^{-1}$ is the inverse of ${T}_N$ and the series converges in $L^p(X)$, $1<p<\infty$.
Using this Calder\'on-type reproducing formula, they provided
the Littlewood--Paley theory for $L^p(X)$, $1 < p < \infty$. Namely,
for each $1 < p < \infty$, there exists a positive constant $C_p$ such that
$$C_p^{-1}\|f\|_p \le \bigg\|\Big\{\sum_{k}|{D}_k(f)|^2\Big\}^{1/2}\bigg\|_p \le C_p\| f\|_p.$$
The above estimates were the key tool in \cite{DJS} for proving the
$T(b)$ theorem on $(X, d, \mu).$

In \cite{HS}, the Besov space was developed via the Littlewood-Paley theory on spaces of homogeneous type $(X, d, \mu)$
with the regularity property \eqref{smooth metric} on the metric $d$ and property \eqref{regular} on the measure $\mu.$
They first introduced a space of test function $\mathcal M(X)$,
and then proved that $R_N$ defined as above in Coifman's decomposition of the identity is bounded on $\mathcal M(X)$
with the operator norm less than 1 for a fixed large integer $N.$
They showed that, for a fixed large integer $N$ and for each $k,$ ${T}_N^{-1}{D}_k^N$ is a test function;
that is, it satisfies similar conditions as ${D}_k$ does. Therefore, they obtained the following Calder\'on-type reproducing formula.
Let $\{{S}_k\}_{k=-\infty}^\infty$ be any approximation to the identity as in \cite {HS} and ${D}_k={S}_k-{S}_{k-1}.$
There exist families of operators $\{{\widetilde {D}}_k\}_{k=-\infty}^\infty$ and $\{{\widetilde {\widetilde {D}}}_k\}_{k=-\infty}^\infty$ such that
\begin{eqnarray}\label{cal3}
f=\sum_{k=-\infty}^\infty {\widetilde {D}}_k{ D}_k(f)=\sum_{k=-\infty}^\infty {D}_k{\widetilde {\widetilde {D}}}_k(f),
\end{eqnarray}
where the series converges in the space $L^p(X), 1<p<\infty,$ the space $\mathcal M(X)$, and the dual $(\mathcal M(X))^\prime$ of $\mathcal M(X)$.

Note that the formula \eqref{cal3} is similar to \eqref{cal1}.
Thus, the theory of Besov spaces on spaces of homogeneous type $(X, d, \mu)$ with properties \eqref{smooth metric} and \eqref{regular} can be developed as in the case of $\Bbb R^n.$
More precisely, the Besov space on such a space of homogeneous type $(X, d, \mu),$ $\dot{B}^{\alpha,q}_p(X)$ for $1\leq p, q\leq \infty$
and $|\alpha|<\theta,$ where $\theta$ depends on the regularity of the approximation to the identity $S_k,$
is defined to be the collection of all $f\in (\mathcal M(X))^\prime$ such that
$$\|f\|_{\dot{B}^{\alpha,q}_p(X)}=\bigg(\sum_{k\in \Bbb Z} \Big(2^{k\alpha}\|D_k(f)\|_p\Big)^q\bigg)^{1/q}<\infty.$$
Again, applying formula \eqref{cal3}, one can show that Besov spaces $\dot{B}^{\alpha,q}_p(X)$ are independent of the choice of approximations to the identity $\{{\rm S}_k\}$ and, moreover, all properties such as embedding, interpolation, duality, atomic decomposition and the $T1$ theorem were obtained (see \cite{HS, H1, H2, HL}).
If the quasi-metric $d$ satisfies \eqref{smooth metric} but the measure $\mu$ satisfies the doubling and the additional reverse doubling condition; that is, there are constants $\kappa \in (0, d]$ and $c \in (0, 1]$ such that
$$
c \lambda^\kappa \mu ( B_d (x, r) ) \leq \mu ( B_d (x, \lambda r) )
$$
for all $x \in X$, $0 < r < \displaystyle\sup_{x, y \in X} d (x, y) / 2$
and $1\leq \lambda < \displaystyle\sup_{x, y \in X} d (x, y) / 2r,$ the theory of the Besov space can be also established.
The key point is that, when $\mu$ satisfies the doubling and the reverse doubling conditions, one can still introduce test function spaces and distributions; moreover, the formula \eqref{cal3} still holds on
$L^p, 1<p<\infty$, test function spaces and distributions. See \cite{HMY1, HMY2} for more details

We now return to the current situation in this paper.
As mentioned, $(\Bbb R^n, \bar \rho, \mu)$ is space of homogeneous type in the sense of Coifman and Weiss.
Note that the quasi-metric $\bar \rho(x,y)$ may have no regularity and the measure $\mu$ only satisfies the doubling property.
Therefore, the method mentioned above can not be carried over to our situation. To achieve our goal, a new approach  is required.

The departure of our new approach is the following result proved by Mac\'ias and Segovia in \cite{MS1}.

\begin{thm}
Let $d(x,y)$ be a quasi-metric on a set $X$. There exists a
quasi-metric $d'(x,y)$ on $X$ such that
\begin{enumerate}
\item[(i)] $d'(x,y)$ is geometrically equivalent to $d(x,y);$ that is, $C^{-1} d(x,y)\le d'(x,y)\le Cd(x,y)$ for some constant $C>0$ and for all $x, y\in X;$
\item[(ii)] $d'$ satisfies the regularity property \eqref{smooth metric}.
\end{enumerate}
\end{thm}

Based on the above theorem, we may assume that $(\Bbb R^n, \bar \rho, \mu)$ is a space of homogeneous type
where the quasi-metric  $\bar \rho$ satisfies the regularity condition \eqref{smooth metric} and the measure $\mu$ satisfies the doubling property.
Under these assumptions, applying Coifman's idea, we still can construct the approximation to the identity associated with $\mathcal F$ (see Lemma 2.1 below for the existence).
We first give the definition as follows. Here and throughout this paper, $V_k(x)$ always denotes the measure $\mu(S(x,2^{-k}))$ for $k\in \Bbb Z$ and $x\in \Bbb R^n$.

\begin{defn}\rm
Let $\bar \rho$ and $\varepsilon$  satisfy condition \eqref{smooth metric}.
A sequence of operators $\{S_k\}_{k\in \Bbb Z}$ is said to be an {\it approximation to the identity associated with $\mathcal F$}
 if there exists a constant $C > 0$ such that, for all $k\in\mathbb{Z}$ and all $x$, $x'$, $y$, $y'\in \Bbb R^n$, the kernels $S_k(x,y)$ of $S_k$ satisfy the following conditions:
\begin{enumerate}
\item[(i)] $S_k(x,y)=0$ if ${\bar \rho}(x,y)> C2^{-k}$ (which means that each $S_k(\cdot, y)$ is supported on the section $S(y, C2^{-k})$ and
    each $S_k(x, \cdot)$ is supported on the section $S(x, C2^{-k})$); \vskip 0.2cm
\item[(ii)] $\displaystyle |S_k(x,y)|\le \frac C{V_k(x)+V_k(y)}$; \vskip 0.2cm
\item[(iii)] $\displaystyle |S_k(x,y)-S_k(x',y)|\le C\frac{(2^k{\bar \rho}(x,x'))^\varepsilon}{V_k(x)+V_k(y)}$\qquad for ${\bar \rho}(x,x')\le C2^{-k}$; \vskip 0.2cm
\item[(iv)] $\displaystyle |S_k(x,y)-S_k(x,y')|\le C\frac{(2^k{\bar \rho}(y,y'))^\varepsilon}{V_k(x)+V_k(y)}$\qquad for ${\bar \rho}(y,y')\le C2^{-k}$; \vskip 0.2cm
\item[(v)] $\displaystyle \big|[S_k(x,y)-S_k(x',y)]-[S_k(x,y')-S_k(x',y')]\big|\le C\frac{(2^k{\bar \rho}(x,x'))^\varepsilon(2^k{\bar \rho}(y,y'))^\varepsilon}{V_k(x)+V_k(y)}$
\item[]   for ${\bar \rho}(x,x')\le C2^{-k}$ and ${\bar \rho}(y,y')\le C2^{-k}$;\vskip 0.2cm
\item[(vi)] $\displaystyle \int_{\Bbb R^n} S_k(x,y)d\mu(x)=1\qquad \text{for all}\ y\in \Bbb R^n$; \vskip 0.2cm
\item[(vii)] $\displaystyle \int_{\Bbb R^n} S_k(x,y)d\mu(y)=1\qquad \text{for all}\ x\in \Bbb R^n.$
\end{enumerate}
\end{defn}

Let $D_k=S_k-S_{k-1}$ and suppose that $\mu(\Bbb R^n)=\infty.$ Applying Coifman's decomposition to the identity yields
$$I=\sum_{k=-\infty}^\infty\sum_{j=-\infty}^\infty D_kD_j=\sum_{|k-j|\leq N} {D}_k{ D}_j+\sum_{|k-j|>N}{D}_k{D}_j:=T_N +R_N.$$
By Cotlar-Stein almost orthogonal estimates, one obtains a similar Calder\'on-type reproducing formula
\begin{eqnarray}\label{cal4}
f=\sum_{k=-\infty}^\infty T_N^{-1}D_k^ND_k(f)=\sum_{k=-\infty}^\infty D_k^ND_kT_N^{-1}(f),
\end{eqnarray}
where, as before, $N$ is a fixed large integer, $D_k^N=\sum_{|j|\leq N}D_{j+k}$ and $T_N^{-1}$ is the inverse of $T_N,$
and the series converges in $L^2(\Bbb R^n, d\mu)$ (see the argument right after Lemma \ref{lem 2.3}).

In this paper we do not consider the $L^p$ convergence for $1<p<\infty$ with $p\ne2$, but we still show
that the above Calder\'on-type reproducing formula \eqref{cal4} holds {\bf for certain subspace of $L^2(\Bbb R^n, d\mu)$}, namely the following

\begin{thm}\label{main1}
Let $\{S_k\}_{k\in \Bbb Z}$ be an approximation to the identity associated with $\mathcal F$ on $(\Bbb R^n, \bar \rho, \mu)$, $\mu(\Bbb R^n)=\infty$,
and $D_k=S_k-S_{k-1}$ for $k\in \Bbb Z$.
For $|\alpha|<\varepsilon/4$ and $1\le p, q\le \infty$, if $f\in L^2(\Bbb R^n, d\mu)$ and satisfies
$$\bigg\{\sum_{k\in \Bbb Z} \Big(2^{k\alpha}\|D_k(f)\|_{L^p_\mu}\Big)^q\bigg\}^{1/q}<\infty,$$
then \eqref{cal4} holds with respect to the norm defined by
$\big\{\sum_{k\in \Bbb Z} \big(2^{k\alpha}\|D_k(f)\|_{L^p_\mu}\big)^q\big\}^{1/q}$,
where we make an appropriate modification for $q=\infty$.
\end{thm}

This result leads to introduce a {\bf new test function space} as follows.

\begin{defn}\rm
Let $\{S_k\}_{k\in \Bbb Z}$ be an approximation to the identity associated with $\mathcal F$ and $D_k=S_k-S_{k-1}$ for $k\in \Bbb Z$.
For $|\alpha|<\varepsilon/4$ and $1\le p,q\le \infty$, define
$$\dot{\mathcal B}^{\alpha,q}_{p,\mathcal F}=\{f\in L^2(\Bbb R^n, d\mu) : \|f\|_{\dot{\mathcal B}^{\alpha,q}_{p,\mathcal F}}<\infty\},$$
where
$$\|f\|_{\dot{\mathcal B}^{\alpha,q}_{p,\mathcal F}}
=\begin{cases} \displaystyle \bigg(\sum_{k\in \Bbb Z} \Big(2^{k\alpha}\|D_k(f)\|_{L^p_\mu}\Big)^q\bigg)^{1/q} & \text{if}\ 1\le q<\infty \\
          \displaystyle \sup_{k\in \Bbb Z} 2^{k\alpha}\|D_k(f)\|_{L^p_\mu} &\text{if}\ q=\infty \end{cases}.$$
\end{defn}

It is clear that the {\bf test function space $\dot{\mathcal B}^{\alpha,q}_{p,\mathcal F}$} is a subspace of $L^2(\Bbb R^n,d\mu).$
Applying the above Calder\'on-type reproducing formula in \eqref{cal4}, one can show that the test function space
$\dot{\mathcal B}^{\alpha,q}_{p,\mathcal F}$ is independent of the choice of the approximation to the identity (see Proposition \ref{prop 3.2} below).
Let {\bf $\big(\dot{\mathcal B}^{\alpha,q}_{p,\mathcal F}\big)'$ denote the distribution space}
(dual of $\dot{\mathcal B}^{\alpha,q}_{p,\mathcal F})$.
Note that for each fixed $k$ and $x$, the function $D_k(x,\cdot)$ belongs to $\dot{\mathcal B}^{\alpha,q}_{p,\mathcal F}$ for all $|\alpha|<\varepsilon/4$, $1\le p,q\le \infty$,
and thus $D_k(f)$ is well defined for all $f\in (\dot{\mathcal B}^{\alpha,q}_{p,\mathcal F})'$(See the proof in section 4.)
 Moreover, applying the second difference smoothness condition of the approximation to the identity associated with $\mathcal F$,
we will show that the Caldero\'n-type reproducing formula \eqref{cal4} still holds on the distribution (dual) space
as follows.

\begin{thm}\label{main2}
Under the same assumptions as Theorem \ref{main1}, for each $f\in \big(\dot{\mathcal B}^{\alpha,q}_{p,\mathcal F} \big)'$,
\begin{equation}\label{1.13}
\langle f, g\rangle =\sum_{k\in \Bbb Z} \langle T_N^{-1}D_kD_k^N(f), g \rangle =\sum_{k\in \Bbb Z} \langle  D_kD_k^NT_N^{-1}(f), g \rangle, \qquad\forall\ g\in \dot{\mathcal B}^{\alpha,q}_{p,\mathcal F}.
\end{equation}
\end{thm}

Once this reproducing formula is established, we can define the Besov space $\dot B^{\alpha,q}_{p,\mathcal F}$ as follows.

\begin{defn}\rm
For $|\alpha|<\varepsilon/4$ and $1\le p,q\le \infty$, let $p'$ and $q'$ denote the conjugate index of $p$ and $q$, respectively.
Suppose that $\{S_k\}_{k\in \Bbb Z}$ is an approximation to the identity associated with $\mathcal F$ on $(\Bbb R^n, \bar \rho, \mu)$ and set $D_k=S_k-S_{k-1}$.
The {\it Besov spaces associated with $\mathcal F$} are defined to be
$$\dot B^{\alpha,q}_{p,\mathcal F}=\Bigg\{f\in \big(\dot{\mathcal B}^{-\alpha,q'}_{p',\mathcal F}\big)': \|f\|_{\dot B^{\alpha,q}_{p,\mathcal F}}
:= \bigg\{\sum_{k\in \Bbb Z} \Big(2^{k\alpha}\|D_k(f)\|_{L^p_\mu}\Big)^q\bigg\}^{1/q}<\infty\Bigg\}$$
with an appropriate modification for $q=\infty$.
\end{defn}

Again, applying the reproducing formula for distribution spaces, we can develop a theory of the Besov spaces on $(\Bbb R^n, \bar \rho, \mu).$
The main result of this theory is the following

\begin{thm}\label{thm 4.1}
Let $\epsilon_1$ be the constant given in condition {\rm ({\bf A})}, $\gamma$ be the constant given in conditions {\rm ({\bf D}6)} and {\rm ({\bf D}7)},
and $\varepsilon$ be the regularity exponent given in \eqref{smooth metric}.
For $|\alpha|<\min\{\varepsilon, \gamma\epsilon_1\}/4$ and $1\le p,q\le \infty$, the Monge-Amp\`ere singular integral operator $H$ is bounded on $\dot B^{\alpha,q}_{p,\mathcal F}$.
\end{thm}

We construct the approximation to the identity associated to sections and obtain the almost orthogonality estimate in the next section.
In section 3 the proofs of Calder\'on-type reproducing formulae on test function spaces $\dot{\mathcal B}^{\alpha,q}_{p,\mathcal F}$ and their duals are given.
We discuss the dense subspaces of Besov spaces $\dot{B}^{\alpha,q}_{p,\mathcal F}$ and their dual spaces as well in section 4.
 Finally, Theorem \ref{thm 4.1} is proved in section 5.

Throughout this paper $C$ denotes a constant not necessarily the same
  at each occurrence, and a subscript is added when we wish to make clear
  its dependence on the parameter. We also use $a \wedge b$ and $a\vee b$ to denote $\min\{a, b\}$ and $\max\{a, b\}$ respectively.
We also write $a\lesssim b$ to indicate that $a$ is majorized by $b$ times a constant independent of $a$ and $b$,
while the notation $a\approx b$ denotes both $a\lesssim b$ and $b\lesssim a$.
%  and $a\approx b$ to denote the equivalence of $a$ and $b$; that is,
%  there exist two positive constants $C_1, C_2$ independent of $a, b$
%  such that $C_1 a \le b \le C_2 a$.

%%%%%%%%%%%%%%%%%%%%%%%%%%%%%%%%%%%%%%%%%%%%%%
%%%%%%%%%%%%%%
%%%%%%%%%%%%%%%%%%%%%%%%%%%%%%%%%%%%%%%%%%%%%%
\section{Existence of the approximation to the identity}\label{approx-idenity}

In this section, we construct the approximation to the identity associated to sections deduced by $\bar\rho$ and $\mu$.
Let $\psi: \mathbb R\mapsto [0,1]$ be a smooth function which is 1 on $(-1, 1)$ and vanishes on $(-\infty, -2)\cup (2, \infty)$.
We define
$$T_k(f)(x)=\int_{\Bbb R^n} \psi(2^k \bar \rho(x,y))f(y)d\mu(y),\quad k\in \Bbb Z.$$
Then
$$T_k(1)(x)\le \int_{\bar \rho(x,y)\le 2^{1-k}}d\mu(y)\le \mu(S(x, 2^{1-k}))\le C\mu(S(x,2^{-k})).$$
Conversely,
$$T_k(1)(x)\ge \int_{\bar \rho(x,y)< 2^{-k}}d\mu(y)=\mu(S(x,2^{-k})).$$
Hence, $T_k(1)(x)\approx \mu(S(x,2^{-k})):=V_k(x)$.
It is easy to check
$V_k(x)\approx V_k(y)$ whenever ${\bar \rho}(x,y)\le  (A_0)^32^{5-k}$. Thus,
\begin{align*}
T_k\bigg(\frac1{T_k(1)}\bigg)(x)
   &=\int_{\Bbb R^n} \psi(2^k {\bar \rho}(x,y))\frac1{T_k(1)(y)}d\mu(y) \\
   &\approx\int_{\Bbb R^n} \psi(2^k {\bar \rho}(x,y)) \frac 1{V_k(y)}d\mu(y) \\
   &\approx \frac 1{V_k(x)}\int_{\Bbb R^n} \psi(2^k {\bar \rho}(x,y))d\mu(y) \\
   &= \frac 1{V_k(x)}T_k(1)(x)\approx 1.
\end{align*}
Let $M_k$ be the operator of multiplication by $M_k(x):=\frac1{T_k(1)(x)}$ and let $W_k$ be operator of multiplication by $W_k(x):=\big[T_k\big(\frac1{T_k(1)}\big)(x)\big]^{-1}$.
We set $S_k=M_kT_kW_kT_kM_k$. Then
%$\lim_{k\to \infty} S_k= \text{identity}$, $\lim_{k\to -\infty} S_k= 0$ and
the kernel of $S_k$ is
$$S_k(x,y)=\int_{\Bbb R^n} M_k(x)\psi(2^k{\bar \rho}(x,z))W_k(z)\psi(2^k{\bar \rho}(z,y))M_k(y)d\mu(z).$$

\begin{lem}\label{id}
There exists a sequence of operators of $\{S_k\}_{k\in \Bbb Z}$ with kernels $S_k(x,y)$ defined on $\Bbb R^n\times \Bbb R^n$ such that the following properties hold:
\begin{enumerate}
\item[(i)] $S_k(x,y)=S_k(y,x);$
\item[(ii)] $S_k(x,y)=0$ if ${\bar \rho}(x,y)> A_02^{2-k}$\ \ and\ \ $\displaystyle |S_k(x,y)|\le \frac C{V_k(x)+V_k(y)},$ where $A_0$ is the constant in \eqref{eqn:quasitriangleineq}$;$
\item[(iii)] $\displaystyle |S_k(x,y)-S_k(x',y)|\le C\frac{(2^k{\bar \rho}(x,x'))^\varepsilon}{V_k(x)+V_k(y)}$\qquad for ${\bar \rho}(x,x')\le (A_0)^32^{5-k};$
\item[(iv)] $\displaystyle |S_k(x,y)-S_k(x,y')|\le C\frac{(2^k{\bar \rho}(y,y'))^\varepsilon}{V_k(x)+V_k(y)}$\qquad for ${\bar \rho}(y,y')\le (A_0)^32^{5-k};$
\item[(v)] $\displaystyle \big|[S_k(x,y)-S_k(x',y)]-[S_k(x,y')-S_k(x',y')]\big|\le C\frac{(2^k{\bar \rho}(x,x'))^\varepsilon(2^k{\bar \rho}(y,y'))^\varepsilon}{V_k(x)+V_k(y)}$
\item[] for ${\bar \rho}(x,x')\le (A_0)^32^{5-k}$ and ${\bar \rho}(y,y')\le (A_0)^32^{5-k};$
\item[(vi)] $\displaystyle \int_{\Bbb R^n} S_k(x,y)d\mu(x)=1\qquad \text{for all}\ y\in \Bbb R^n;$
\item[(vii)] $\displaystyle \int_{\Bbb R^n} S_k(x,y)d\mu(y)=1\qquad \text{for all}\ x\in \Bbb R^n.$
\end{enumerate}
\end{lem}

\begin{proof}
Property (i) is obvious since $\bar \rho(x,y)=\bar \rho(y,x)$.
(ii) If $S_k(x,y)\neq 0$, then ${\bar \rho}(x,z)\le 2^{1-k}$ and ${\bar \rho}(z,y)\le 2^{1-k}$.
Hence ${\bar \rho}(x,y)\le A_02^{2-k}$.
That is, $S_k(x,y)=0$ when ${\bar \rho}(x,y)> A_02^{2-k}$. The definition of $M_k$ gives
\begin{align*}
|S_k(x,y)|
   &\le \frac 1{T_k(1)(x)}\frac 1{T_k(1)(y)}\int_{{\bar \rho}(x,z)\le 2^{1-k}} \psi(2^k{\bar \rho}(x,z))W_k(z)\psi(2^k{\bar \rho}(z,y))d\mu(z) \\
   &\le C\frac 1{V_k(x)}\frac 1{V_k(y)}\mu(S(x,2^{1-k})) \\
   &\leq\frac C{V_k(y)},
\end{align*}
which implies $|S_k(x,y)|\le \frac C{V_k(x)+V_k(y)}$ whenever ${\bar \rho}(x,y)\le A_02^{2-k}.$

To estimate (iii), we write
\begin{align*}
&S_k(x,y)-S_k(x',y) \\
   &=\int_{\Bbb R^n} [M_k(x)\psi(2^k{\bar \rho}(x,z))-M_k(x')\psi(2^k{\bar \rho}(x',z))]W_k(z)\psi(2^k{\bar \rho}(z,y))M_k(y)d\mu(z) \\
   &= \int_{\Bbb R^n} [M_k(x)-M_k(x')]\psi(2^k{\bar \rho}(x,z))W_k(z)\psi(2^k{\bar \rho}(z,y))M_k(y)d\mu(z)\\
   &\qquad + \int_{\Bbb R^n} M_k(x')[\psi(2^k{\bar \rho}(x,z))-\psi(2^k{\bar \rho}(x',z))]W_k(z)\psi(2^k{\bar \rho}(z,y))M_k(y)d\mu(z)\\
   &:= I_1+I_2.
\end{align*}
For $I_1$, we have
$$|M_k(x)-M_k(x')|=\frac {|T_k(1)(x')-T_k(1)(x)|}{T_k(1)(x')T_k(1)(x)}\approx \frac {|T_k(1)(x')-T_k(1)(x)|}{V_k(x')V_k(x)}.$$
Since $\vert {\bar \rho} (x, z) - {\bar \rho} (y, z)\vert \leq C({\bar \rho} (x,y))^\varepsilon
  [{\bar \rho} (x, z) +{\bar \rho} (y, z)]^{1-\varepsilon}$, we have
\begin{align}\label{eq 2.1}
|\psi(2^k {\bar \rho}(x,y))-\psi(2^k {\bar \rho}(x',y))|
&\le C2^k({\bar \rho} (x,x'))^\varepsilon
           [{\bar \rho} (x, y) +{\bar \rho} (x', y)]^{1-\varepsilon} \nonumber  \\
&\le C2^k2^{-k(1-\varepsilon)}({\bar \rho} (x,x'))^\varepsilon\\
&=C(2^k{\bar \rho} (x,x'))^\varepsilon\qquad\text{for}\ {\bar \rho}(x,x')\le (A_0)^32^{5-k}.\nonumber
\end{align}
Then for ${\bar \rho}(x,x')\le (A_0)^32^{5-k}$,
$$|T_k(1)(x')-T_k(1)(x)|\le \int_{\Bbb R^n} |\psi(2^k {\bar \rho}(x,y))-\psi(2^k {\bar \rho}(x',y))|d\mu(y)\le CV_k(x')(2^k{\bar \rho} (x,x'))^\varepsilon,$$
which yields
\begin{equation}\label{eq 2.2}
|M_k(x)-M_k(x')|\le C(2^k{\bar \rho} (x,x'))^\varepsilon\frac 1{V_k(x)}.
\end{equation}
Therefore,
$$|I_1|\le C(2^k{\bar \rho} (x,x'))^\varepsilon\frac 1{V_k(x)}
\le C(2^k{\bar \rho} (x,x'))^\varepsilon\frac 1{V_k(x)+V_k(y)}\qquad\mbox{for}\ {\bar \rho}(x,x')\le (A_0)^32^{5-k}.$$
For ${\bar \rho}(x,x')\le (A_0)^32^{5-k}$, it follows from \eqref{eq 2.1} that
\begin{align*}
|I_2|
&\le \int_{\Bbb R^n} |M_k(x')[\psi(2^k{\bar \rho}(x,z))-\psi(2^k{\bar \rho}(x',z))]W_k(z)\psi(2^k{\bar \rho}(z,y))M_k(y)|d\mu(z)\\
&\le C(2^k{\bar \rho} (x,x'))^\varepsilon\frac 1{V_k(y)}\le C(2^k{\bar \rho} (x,x'))^\varepsilon\frac 1{V_k(x)+V_k(y)}.
\end{align*}

The proof of (iv) is similar to (iii).

To verify (v), we write
\begin{align*}
&[S_k(x,y)-S_k(x',y)]-[S_k(x,y')-S_k(x',y')] \\
&=\int_{\Bbb R^n} [M_k(x)\psi(2^k{\bar \rho}(x,z))-M_k(x')\psi(2^k{\bar \rho}(x',z))]W_k(z)\\
&\hskip 1cm\times [\psi(2^k{\bar \rho}(z,y))M_k(y)-\psi(2^k{\bar \rho}(z,y'))M_k(y')]d\mu(z) \\
&=\int_{\Bbb R^n} [M_k(x)-M_k(x')]\psi(2^k{\bar \rho}(x,z))W_k(z)[\psi(2^k{\bar \rho}(z,y))-\psi(2^k{\bar \rho}(z,y'))]M_k(y)d\mu(z) \\
&\qquad+ \int_{\Bbb R^n} [M_k(x)-M_k(x')]\psi(2^k{\bar \rho}(x,z))W_k(z)\psi(2^k{\bar \rho}(z,y'))[M_k(y)-M_k(y')]d\mu(z) \\
&\qquad+ \int_{\Bbb R^n} M_k(x')[\psi(2^k{\bar \rho}(x,z))-\psi(2^k{\bar \rho}(x',z))]W_k(z)\\
&\hskip 2cm\times [\psi(2^k{\bar \rho}(z,y))-\psi(2^k{\bar \rho}(z,y'))]M_k(y)d\mu(z) \\
&\qquad+ \int_{\Bbb R^n} M_k(x')[\psi(2^k{\bar \rho}(x,z))-\psi(2^k{\bar \rho}(x',z))]W_k(z)\psi(2^k{\bar \rho}(z,y'))[M_k(y)-M_k(y')]d\mu(z) \\
&:=J_1+J_2+J_3+J_4.
\end{align*}
To estimate $J_1$, we use \eqref{eq 2.1} and \eqref{eq 2.2} for ${\bar \rho}(x,x')\le (A_0)^32^{5-k}$ and ${\bar \rho}(y,y')\le (A_0)^32^{5-k}$ combined with the
support condition of $\psi$ to get
$$|J_1|\le C(2^k{\bar \rho} (x,x'))^\varepsilon(2^k{\bar \rho} (y,y'))^\varepsilon\frac 1{V_k(x)+V_k(y)}.$$
Similarly, for
${\bar \rho}(x,x')\le (A_0)^32^{5-k}$ and ${\bar \rho}(y,y')\le (A_0)^32^{5-k}$,
$$|J_2|+|J_3|+|J_4|\le C(2^k{\bar \rho} (x,x'))^\varepsilon(2^k{\bar \rho} (y,y'))^\varepsilon\frac 1{V_k(x)+V_k(y)}.$$

For (vi), we have
\begin{align*}
\int S_k(x,y)d\mu(x)
&=\iint M_k(x)\psi(2^k{\bar \rho}(x,z))W_k(z)\psi(2^k{\bar \rho}(z,y))M_k(y)d\mu(z)d\mu(x) \\
&=\int \bigg(\int \psi(2^k{\bar \rho}(z,x))M_k(x)d\mu(x)\bigg)W_k(z)\psi(2^k{\bar \rho}(z,y))M_k(y)d\mu(z) \\
&=\int \bigg[T_k\Big(\frac1{T_k(1)}\Big)(z)\bigg]W_k(z)\psi(2^k{\bar \rho}(z,y))M_k(y)d\mu(z) \\
&=M_k(y) \int \psi(2^k{\bar \rho}(z,y))d\mu(z) \\
&=M_k(y)T_k(1)(y)=1,
\end{align*}
and (vii) is obtained by the same argument.
\end{proof}

\begin{lem}\label{lem 2.2}
Let $\{S_k\}_{k\in \Bbb Z}$ be an approximation to the identity associated with $\mathcal F$
and set $D_k=S_k-S_{k-1}$ for $k\in \Bbb Z$. There exists a constant $C$ such that
\begin{equation*}
|D_jD_k(x,y)|\le C2^{-|j-k|\varepsilon}\frac 1{V_{\min\{j,k\}}(x)+V_{\min\{j,k\}}(y)}.
\end{equation*}
\end{lem}

\begin{proof}
For $k\ge j$, we use vanishing condition of $D_k$ and Lemma \ref{id} (ii), (iv) to get
\begin{align*}
|D_jD_k(x,y)|
&\le \int_{{\bar \rho}(y,z)\le A_02^{3-k}} |D_j(x,z)-D_j(x,y)||D_k(z,y)|d\mu(z) \\
&\le C\int_{{\bar \rho}(y,z)\le A_02^{3-k}} \Big(2^j{\bar \rho}(z,y)\Big)^\varepsilon \frac1{V_j(y)}\frac1{V_k(y)}d\mu(z)\\
&\le C2^{-(k-j)\varepsilon}\frac1{V_j(y)}.
\end{align*}
Similarly, for $k<j$, the vanishing condition of $D_j$  and Lemma \ref{id} (ii), (iii) show
\begin{align*}
|D_jD_k(x,y)|
&\le \int_{{\bar \rho}(x,z)\le A_02^{3-j}} |D_j(x,z)||D_k(z,y)-D_k(x,y)|d\mu(z) \\
&\le C\int_{{\bar \rho}(x,z)\le A_02^{3-j}} \frac1{V_j(x)}\Big(2^k{\bar \rho}(z,x)\Big)^\varepsilon \frac1{V_k(x)}d\mu(z)\\
&\le C2^{-(j-k)\varepsilon}\frac1{V_k(x)}.
\end{align*}
Since $V_k(x)\approx V_k(y)$ when ${\bar \rho}(x,y)\le (A_0)^22^{4-k}$,
the proof is finished.
\end{proof}

By Lemma \ref{id} (ii) and Lemma \ref{lem 2.2}, we immediately have the following result.

\begin{lem}\label{lem 2.3}
Let $\{S_k\}_{k\in \Bbb Z}$ be an approximation to the identity associated with $\mathcal F$
and set $D_k=S_k-S_{k-1}$ for $k\in \Bbb Z$. For $1\le p\le \infty$, there exists a constant $C$ such that
$$\|D_jD_k\|_{L^p_\mu \mapsto L^p_\mu}\le C2^{-|j-k|\varepsilon}.$$
\end{lem}

Plugging $p=2$ into Lemma \ref{lem 2.3}, the Cotlar-Stein lemma says
$$\|R_N(f)\|_{L^2_\mu}\le C2^{-N\varepsilon}\|f\|_{L^2_\mu}$$
and then $T_N^{-1}$ is bounded on $L^2_\mu$. This yields
$$I=\sum_{k\in \Bbb Z} T_N^{-1}D_k^ND_k=\sum_{k\in \Bbb Z} D_k^ND_kT_N^{-1} \qquad\mbox{in}\ L^2_\mu,$$
which is  \eqref{cal4}.
%%%%%%%%%%%%%%%%%%%%%%%%%%%%%%%%%%%%%%%%%%%%%%
%%%%%%%%%%%%%%
%%%%%%%%%%%%%%%%%%%%%%%%%%%%%%%%%%%%%%%%%%%%%%
\section{Calder\'on-type reproducing formulae for $\dot{\mathcal B}^{\alpha,q}_{p,\mathcal F}$ and its dual}\label{besov}

In this section, we show Theorems \ref{main1} and \ref{main2},
which are the Calder\'on-type reproducing formula for $\dot{\mathcal B}^{\alpha,q}_{p,\mathcal F}$ and its dual, respectively.

\begin{proof}[Proof of Theorem \ref{main1}]
We prove the first equality in (\ref{cal4}) only because the proof for the second one is similar.
We first show that if $N$ is chosen to be large enough then there exists a constant $C$ such that
\begin{equation}\label{3.002}
\|R_N(f)\|_{\dot{\mathcal B}^{\alpha,q}_{p,\mathcal F}}
    \le CN^{\frac32}2^{-N(\frac \varepsilon2-2|\alpha|)}\|f\|_{\dot{\mathcal B}^{\alpha,q}_{p,\mathcal F}}.
\end{equation}
To do this, since $R_N$ is bounded on $L^2(\Bbb R^n,d\mu)$ and $f=\sum_{k\in \Bbb Z} T_N^{-1}D_k^ND_k(f)$ in $L^2(\Bbb R^n,d\mu)$, we write
\begin{align*}
\|R_N(f)\|_{\dot{\mathcal B}^{\alpha,q}_{p,\mathcal F}}
&=\bigg\{\sum_{k\in \Bbb Z} \Big(2^{k\alpha}\big\|D_kR_N(f)\big\|_{L^p_\mu}\Big)^q\bigg\}^{1/q}\\
&=\bigg\{\sum_{k\in \Bbb Z} \bigg(2^{k\alpha}\Big\|D_kR_N\Big(\sum_{k'\in \Bbb Z} T_N^{-1}D_{k'}^ND_{k'}(f)\Big)\Big\|_{L^p_\mu}\bigg)^q\bigg\}^{1/q}.
\end{align*}
Observing
\begin{equation}\label{3.2}
\begin{split}
D_kR_N\Big(\sum_{k'} T_N^{-1} D^N_{k'}D_{k'}(f)\Big)(x)
&=D_kR_N \sum_{k'}\sum_{m=0}^\infty (R_N)^m D^N_{k'}D_{k'}(f)(x) \\
&=\sum_{k'}\sum_{m=0}^\infty D_k(R_N)^{m+1} D^N_{k'}D_{k'}(f)(x)
\end{split}
\end{equation}
and plugging  $R_N=\sum_{|k-\ell|>N} D_kD_\ell$ yield
\begin{align*}
D_k(R_N)^{m+1}D^N_{k'}
&=D_k\sum_{|k_0-\ell_0|>N} D_{k_0}D_{\ell_0}\sum_{|k_1-\ell_1|>N} D_{k_1}D_{\ell_1}\cdots\sum_{|k_m-\ell_m|>N} D_{k_m}D_{\ell_m}D^N_{k'}\\
&=\sum_{|k_0-\ell_0|>N}\cdots \sum_{|k_m-\ell_m|>N} D_kD_{k_0}D_{\ell_0}D_{k_1}D_{\ell_1}\cdots D_{k_m}D_{\ell_m}D^N_{k'}.
\end{align*}
Thus
\begin{align*}
\|R_N(f)\|_{\dot{\mathcal B}^{\alpha,q}_{p,\mathcal F}}
%&=\bigg\{\sum_{k\in \Bbb Z} \Big(2^{k\alpha}\|D_kR_N(f)\|_{L^p_\mu}\Big)^q\bigg\}^{1/q} \\
%&\le  \bigg\{\sum_{k}\bigg(\sum_{k'}\sum_{m=0}^\infty
%  \sum_{|k_0-\ell_0|>N}\cdots \sum_{|k_m-\ell_m|>N}  \\
%&\qquad\quad  2^{k\alpha} \|D_kD_{k_0}D_{\ell_0}D_{k_1}D_{\ell_1}\cdots D_{k_m}D_{\ell_m}D^N_{k'}D_{k'}(f)\|_{L^p_\mu}\bigg)^q\bigg\}^{1/q}\\
&\le  \bigg\{\sum_{k}\bigg(\sum_{k'}
\sum_{m=0}^\infty  \sum_{|k_0-\ell_0|>N}\cdots \sum_{|k_m-\ell_m|>N}  \\
&\qquad\quad 2^{k\alpha}\big\|D_kD_{k_0}D_{\ell_0}D_{k_1}D_{\ell_1}\cdots D_{k_m}D_{\ell_m}D^N_{k'}\big\|_{L^p_\mu\mapsto L^p_\mu}
\|D_{k'}(f)\|_{L^p_\mu}\bigg)^q\bigg\}^{1/q}.
\end{align*}

Note that $D^N_{k'}= \sum_{|j|\le N}D_{k'+j}.$ Applying Lemma \ref{lem 2.3} gives that
\begin{align*}
&\big\|D_kD_{k_0}D_{\ell_0}D_{k_1}D_{\ell_1}\cdots D_{k_m}D_{\ell_m}D^N_{k'}\big\|_{L^p_\mu\mapsto L^p_\mu} \\
&\qquad\le \sum_{|j|\le N} \|D_kD_{k_0}D_{\ell_0}D_{k_1}D_{\ell_1}\cdots D_{k_m}D_{\ell_m}D_{k'+j}\|_{L^p_\mu\mapsto L^p_\mu} \\
&\qquad\le  C\sum_{|j|\le N}2^{-|k-k_0|\varepsilon}2^{-|\ell_0-k_1|\varepsilon}\cdots2^{-|\ell_{m-1}-k_m|\varepsilon}2^{-|\ell_m-k'-j|\varepsilon}
\end{align*}
and
\begin{align*}
\big\|D_kD_{k_0}D_{\ell_0}D_{k_1}D_{\ell_1}\cdots D_{k_m}D_{\ell_m}D^N_{k'}\big\|_{L^p_\mu\mapsto L^p_\mu}
&\le CN\|D_{k_0}D_{\ell_0}D_{k_1}D_{\ell_1}\cdots D_{k_m}D_{\ell_m}\|_{L^p_\mu\mapsto L^p_\mu} \\
&\le CN2^{-|k_0-\ell_0|\varepsilon}2^{-|k_1-\ell_1|\varepsilon}\cdots2^{-|k_m-\ell_m|\varepsilon}.
\end{align*}
Taking an average of these two estimates yields
\begin{align*}
&\big\|D_kD_{k_0}D_{\ell_0}D_{k_1}D_{\ell_1}\cdots D_{k_m}D_{\ell_m}D^N_{k'}\big\|_{L^p_\mu\mapsto L^p_\mu} \\
&\le CN^{\frac12}\sum_{|j|\le N}2^{-|k-k_0|\varepsilon/2}2^{-|k_0-\ell_0|\varepsilon/2}
         2^{-|\ell_0-k_1|\varepsilon/2}\cdots2^{-|\ell_{m-1}-k_m|\varepsilon/2}2^{-|k_m-\ell_m|\varepsilon/2}2^{-|\ell_m-k'-j|\varepsilon/2}.
\end{align*}
Inserting
\begin{align*}
&2^{k\alpha}=2^{(k-k_0)\alpha}2^{(k_0-\ell_0)\alpha}
	2^{(\ell_0-k_1)\alpha}\cdots2^{(\ell_{m-1}-k_m)\alpha}2^{(k_m-\ell_m)\alpha}
		2^{(\ell_m-k'-j)\alpha}2^{(k'+j)\alpha}
	\end{align*}
into the above last estimate implies
\begin{align*}
&2^{k\alpha}\big\|D_kD_{k_0}D_{\ell_0}D_{k_1}D_{\ell_1}\cdots D_{k_m}D_{\ell_m}D^N_{k'}\big\|_{L^p_\mu\mapsto L^p_\mu} \\
&\qquad\le CN^{\frac12}2^{-|k-k_0|(\varepsilon/2-|\alpha|)}2^{-|k_0-\ell_0|(\varepsilon/2-|\alpha|)}\\
&\qquad\quad\times  2^{-|\ell_0-k_1|(\varepsilon/2-|\alpha|)}\cdots2^{-|k_m-\ell_m|(\varepsilon/2-|\alpha|)}
\sum_{|j|\le N}2^{-|\ell_m-k'-j|(\varepsilon/2-|\alpha|)}2^{(k'+j)\alpha}.
\end{align*}
Hence,
\begin{align}\label{3.3}
\begin{split}
&2^{k\alpha}\big\|D_k(R_N)^{m+1} D^N_{k'}\big\|_{L^p_\mu\mapsto L^p_\mu}\\
&\le CN^{1/2}\sum_{|k_0-\ell_0|>N}\cdots \sum_{|k_m-\ell_m|>N}
2^{-|k-k_0|(\varepsilon/2-|\alpha|)}2^{-|k_0-\ell_0|(\varepsilon/2-|\alpha|)}\\
&\qquad\times2^{-|\ell_0-k_1|(\varepsilon/2-|\alpha|)}\cdots2^{-|k_m-\ell_m|(\varepsilon/2-|\alpha|)}
\sum_{|j|\le N}2^{-|\ell_m-k'-j|(\varepsilon/2-|\alpha|)}2^{(k'+j)\alpha}.
\end{split}
\end{align}
Applying H\"older's inequality gives
\begin{align*}
&\|R_N(f)\|_{\dot{\mathcal B}^{\alpha,q}_{p,\mathcal F}}\\
&\quad\le CN^{1/2} \bigg\{\sum_{k}\bigg(\sum_{k'}
\sum_{m=0}^\infty  \sum_{|k_0-\ell_0|>N}\cdots \sum_{|k_m-\ell_m|>N}\sum_{|j|\leq N} 2^{-|k-k_0|(\varepsilon/2-|\alpha|)}2^{-|k_0-\ell_0|(\varepsilon/2-|\alpha|)} \\
&\hskip 2cm \times 2^{-|\ell_0-k_1|(\varepsilon/2-|\alpha|)}\cdots2^{-|k_m-\ell_m|(\varepsilon/2-|\alpha|)}
2^{-|\ell_m-k'-j|(\varepsilon/2-|\alpha|)}\bigg)^{q/q'} \\
&\hskip 2cm \times \bigg(\sum_{k'}\sum_{m=0}^\infty  \sum_{|k_0-\ell_0|>N}\cdots \sum_{|k_m-\ell_m|>N}\sum_{|j|\le N} 2^{-|k-k_0|(\varepsilon/2-|\alpha|)}2^{-|k_0-\ell_0|(\varepsilon/2-|\alpha|)} \\
&\hskip 2cm \times
         2^{-|\ell_0-k_1|(\varepsilon/2-|\alpha|)}\cdots2^{-|k_m-\ell_m|(\varepsilon/2-|\alpha|)}2^{-|\ell_m-k'-j|(\varepsilon/2-|\alpha|)}\\
&\hskip 2cm \times \Big(2^{(k'+j)\alpha}\|D_{k'}(f)\|_{L^p_\mu}\Big)^q\bigg)\bigg\}^{1/q}.
\end{align*}
Observe that if we choose $N$ large enough so that $2^{-N(\varepsilon/2 -|\alpha|)}<1$ then
\begin{align*}
&\sum_{k'}\sum_{m=0}^\infty  \sum_{|k_0-\ell_0|>N}\cdots \sum_{|k_m-\ell_m|>N}
2^{-|k-k_0|(\varepsilon/2-|\alpha|)}2^{-|k_0-\ell_0|(\varepsilon/2-|\alpha|)}\\
&\hskip 5cm \times  2^{-|\ell_0-k_1|(\varepsilon/2-|\alpha|)}\cdots2^{-|k_m-\ell_m|(\varepsilon/2-|\alpha|)}
 2^{-|\ell_m-k'-j|(\varepsilon/2-|\alpha|)}\\
&\qquad \le C2^{-N(\varepsilon/2 -|\alpha|)}
\end{align*}
and
\begin{align*}
&\sum_{k}\sum_{m=0}^\infty  \sum_{|k_0-\ell_0|>N}\cdots \sum_{|k_m-\ell_m|>N}
2^{-|k-k_0|(\varepsilon/2-|\alpha|)}2^{-|k_0-\ell_0|(\varepsilon/2-|\alpha|)}\\
&\hskip 5cm \times  2^{-|\ell_0-k_1|(\varepsilon/2-|\alpha|)}\cdots2^{-|k_m-\ell_m|(\varepsilon/2-|\alpha|)}
 2^{-|\ell_m-k'-j|(\varepsilon/2-|\alpha|)}\\
&\qquad \le C2^{-N(\varepsilon/2 -|\alpha|)}.
\end{align*}
Note that $\sum_{|j|\le N}2^{j\alpha}\le CN2^{N|\alpha|}.$ Finally, we have
\begin{align*}
\|R_N(f)\|_{\dot{\mathcal B}^{\alpha,q}_{p,\mathcal F}}
&\le CN^{3/2}2^{-N(\varepsilon/2-2|\alpha|)}\bigg\{\sum_{k'\in \Bbb Z}\Big(2^{k'\alpha}\|D_{k'}(f)\|_{L^p_\mu} \Big)^q\bigg\}^{1/q}\\
&\le CN^{3/2}2^{-N(\varepsilon/2-2|\alpha|)}\|f\|_{\dot{\mathcal B}^{\alpha,q}_{p,\mathcal F}}
\end{align*}
which gives the estimate in \eqref{3.002}.

Note that $T_N^{-1}=(I-R_N)^{-1}=\sum_{m=0}^\infty (R_N)^m$ and choose $N$ large enough such that $CN^{3/2}2^{-N(\varepsilon/2-2|\alpha|)}<1$, so \eqref{3.002}  implies
\begin{equation}\label{3.1}
\|T_N^{-1}(f)\|_{\dot{\mathcal B}^{\alpha,q}_{p,\mathcal F}}\le C_N\|f\|_{\dot{\mathcal B}^{\alpha,q}_{p,\mathcal F}},
\end{equation}
which shows that $T^{-1}_N$ is bounded on $\dot{\mathcal B}^{\alpha,q}_{p,\mathcal F}.$
In order to prove that $\sum_{k\in \Bbb Z} T^{-1}_ND^N_kD_k(f)$ converges to $f$ in
$\dot{\mathcal B}^{\alpha,q}_{p,\mathcal F}$, we observe that
$$f(x)- \sum_{|k|\le M} T^{-1}_ND^N_kD_k(f)(x)= \sum_{|k|>M} T^{-1}_ND^N_kD_k(f)(x)\qquad\text{for}\
  f\in L^2_\mu.$$
Therefore, we only need to show
$$\lim_{M\to \infty}\bigg\|\sum_{|k|>M} T^{-1}_ND^N_kD_k(f)\bigg\|_{\dot{\mathcal B}^{\alpha,q}_{p,\mathcal F}}=0.$$
Indeed, by \eqref{3.1},
$$\bigg\|\sum_{|k|>M} T^{-1}_ND^N_kD_k(f)\bigg\|_{\dot{\mathcal B}^{\alpha,q}_{p,\mathcal F}}
\le C_N \bigg\|\sum_{|k|>M} D^N_kD_k(f)\bigg\|_{\dot{\mathcal B}^{\alpha,q}_{p,\mathcal F}}.$$
The same argument as the proof of \eqref{3.002} yields
$$\bigg\|\sum_{|k|>M} T^{-1}_ND^N_kD_k(f)\bigg\|_{\dot{\mathcal B}^{\alpha,q}_{p,\mathcal F}} \le C_N\bigg\{\sum_{|k|>M} \Big(2^{k\alpha}\|D_k(f)\|_{L^p_\mu}\Big)^q\bigg\}^{1/q}.$$
The assumption of $f$ shows that the right hand side of the above inequality goes to $0$ as $M\to \infty$,
and hence the first equality in (\ref{cal4}) holds.
\end{proof}
We now prove Theorem \ref{main2}.
\begin{proof}[Proof of Theorem \ref{main2}]
For $g\in \dot{\mathcal B}^{\alpha,q}_{p,\mathcal F}$ and
$f\in \big(\dot{\mathcal B}^{\alpha,q}_{p,\mathcal F} \big)'$, Theorem \ref{main1} says
\begin{equation}\label{dual}
\langle f,g \rangle=\bigg\langle f, \sum_{k\in \Bbb Z} T^{-1}_ND_k^ND_k(g) \bigg\rangle =  \sum_{k\in \Bbb Z} \langle f, T^{-1}_ND_k^ND_k(g)\rangle.
\end{equation}
Since $S_k$ is self-adjoint (Lemma \ref{id} (i)), operators $D_k$, $D^N_k$ and $T^{-1}_N$ are all self-adjoint. Therefore, it suffices to show
\begin{equation}\label{adj}
\langle f, T^{-1}_ND_k^ND_k(g)\rangle = \langle(D_k)^*(D_k^N)^*(T^{-1}_N)^*(f), g \rangle=\langle D_kD_k^N T^{-1}_N(f), g \rangle.
\end{equation}
Taking the summation for $k\in \Bbb Z$ on both sides of (\ref{adj}) yields the second equality of (\ref{1.13}).
The argument for the first equality in (\ref{1.13}) is similar, and we omit the details.

To give a rigorous proof of \eqref{adj}, we claim
  \begin{equation}\label{3.51}
  \big\langle f, D_{k}(g)\big\rangle=\big\langle D_{k}(f),g\big\rangle
\qquad\text{for}\ g\in \dot{\mathcal B}^{\alpha,q}_{p,\mathcal F},\ f\in \big(\dot{\mathcal B}^{\alpha,q}_{p,\mathcal F} \big)'.
  \end{equation}
Assuming the claim for the moment, we have
\begin{align*}
	\big\langle f, D_{k'+\ell}D_{k'}(R_{N})^{m-1}D_{k}^{N}D_{k}(g)\big\rangle
	&=\big\langle D_{k'+\ell}(f),D_{k'}(R_{N})^{m-1}D_{k}^{N}D_{k}(g)\big\rangle\\
	&=\big\langle D_{k'}D_{k'+\ell}(f),(R_{N})^{m-1}D_{k}^{N}D_{k}(g)\big\rangle.
\end{align*}
Since $R_N$ can be expressed to be $R_N=\sum_{k'\in \mathbb Z}\sum_{|\ell|>N} D_{k'+\ell}D_{k'}=\sum_{k'\in \mathbb Z}\sum_{|\ell|>N}  D_{k'}D_{k'+\ell}$,
we take the summation $\sum_{k'\in \mathbb Z}\sum_{|\ell|>N}$ on both sides to obtain
$$\big\langle f, R_N(R_{N})^{m-1}D_{k}^{N}D_{k}(g)\big\rangle=\big\langle R_{N}(f),(R_{N})^{m-1}D_{k}^{N}D_{k}(g)\big\rangle.$$
Repeating the same process $m$ times, we obtain
$$\big\langle f, T^{-1}_{N}D_{k}^{N}D_{k}(g)\big\rangle=\big\langle T_{N}^{-1}(f),D_{k}^{N}D_{k}(g)\big\rangle$$
and then
$$\big\langle f, T^{-1}_{N}D_{k}^{N}D_{k}(g)\big\rangle=\big\langle D_{k}D_{k}^{N}T_{N}^{-1}(f),g\big\rangle,$$
which and \eqref{dual} give us
$$\langle f,g \rangle= \sum_{k\in \mathbb Z} \big\langle D_{k}D_{k}^{N}T_{N}^{-1}(f),g\big\rangle.$$
The first equality of \eqref{1.13} can be obtained similarly.

We now return to the proof of claim \eqref{3.51}, which contains three steps:
{\leftmargini=1.52cm
\begin{itemize}
\item[{\bf Step 1.}] Show that each $D_k$ is bounded on $\dot{\mathcal B}^{\alpha,q}_{p,\mathcal F}$ for all $|\alpha|<\frac{\varepsilon}{4}$ and $1\leq p,q\leq\infty$.
\item[{\bf Step 2.}] Show that
$\langle f,D_{k}(g)\rangle=\langle D_{k}(f),g\rangle$ for all $f\in(\dot{\mathcal B}^{\alpha,q}_{p,\mathcal F})'$ and
$g\in\dot{\mathcal B}^{\alpha,q}_{p,\mathcal F}\cap L^{p}_{\mu}$.
\item[{\bf Step 3.}] Show that
$\dot{\mathcal B}^{\alpha,q}_{p,\mathcal F}\subset\overline{L^{p}_{\mu}\cap\dot{\mathcal B}^{\alpha,q}_{p,\mathcal F}}$,
where $\overline{L^{p}_{\mu}\cap\dot{\mathcal B}^{\alpha,q}_{p,\mathcal F}}$ denotes the closure of $L^{p}_{\mu}\cap\dot{\mathcal B}^{\alpha,q}_{p,\mathcal F}$ with respect to
$\|\cdot\|_{\dot{\mathcal B}^{\alpha,q}_{p,\mathcal F}}$.
\end{itemize}}

To prove {\bf step 1}, we use Theorem \ref{main1} to write
\begin{align*}
  \|D_{k}(f)\|_{\dot{\mathcal B}^{\alpha,q}_{p,\mathcal F}}
  &=\bigg\{\sum_{\ell\in\mathbb Z}2^{\ell\alpha q}\Big\|D_{\ell}D_{k}\Big(\sum_{k'\in\mathbb Z}D_{k'}^{N}D_{k'}T_{N}^{-1}(f)\Big)\Big\|_{L^{p}_{\mu}}^{q}\bigg\}^{\frac{1}{q}}\\
  &\leq\bigg\{\sum_{\ell\in\mathbb Z}2^{\ell\alpha q}\Big(\sum_{k'\in\mathbb Z}\|D_{\ell}D_{k}D_{k'}^{N}\|_{L^{p}_{\mu}\mapsto L^{p}_{\mu}}
              \|D_{k'}T_{N}^{-1}(f)\|_{L^{p}_{\mu}}\Big)^{q}\bigg\}^{\frac{1}{q}}.
  \end{align*}
By the same argument as the proof of \eqref{3.002},
\begin{align*}
 \|D_{k}(f)\|_{\dot{\mathcal B}^{\alpha,q}_{p,\mathcal F}}
 &\lesssim N^{\frac{1}{2}}\bigg\{\sum_{\ell\in\mathbb Z}2^{\ell\alpha q}\Big(\sum_{k'\in\mathbb Z}\sum_{|j|\leq N}2^{-|\ell-k'-j|\frac{\varepsilon}{2}}
 \|D_{k'}T_{N}^{-1}(f)\|_{L^{p}_{\mu}}\Big)^{q}\bigg\}^{\frac{1}{q}}\\
 &\lesssim N^{\frac{1}{2}}\Big(\sum_{|j|\leq N}2^{j\alpha}\Big)\bigg\{\sum_{k'\in\mathbb Z}
  2^{k'\alpha q}\|D_{k'}T_{N}^{-1}(f)\|_{L^{p}_{\mu}}^{q}\bigg\}^{\frac{1}{q}}\\
 &\lesssim N^{\frac{3}{2}}2^{N|\alpha|}\|T_{N}^{-1}(f)\|_{\dot{\mathcal B}^{\alpha,q}_{p,\mathcal F}}\\
  &\lesssim N^{\frac{3}{2}} 2^{N|\alpha|}C_N \|f\|_{\dot{\mathcal B}^{\alpha,q}_{p,\mathcal F}}.
   \end{align*}

To show {\bf step 2}, for $g\in\dot{\mathcal B}^{\alpha,q}_{p,\mathcal F}\cap L^{p}_{\mu}$, we define
  $$g_{k,M}(x)=\int_{S(0,M)}D_{k}(x,y)g(y)d\mu(y),\qquad M>0,$$
  where $S(0,M)$ denotes the section $\{y\in\mathbb R^n:\bar \rho(0,y)<M\}$. By step 1,
  $$\|D_{k}(g)-g_{k,M}\|_{\dot{\mathcal B}^{\alpha,q}_{p,\mathcal F}}=\|D_{k}(g\chi_{\mathbb R^{n}\setminus S(0,M)})\|_{\dot{\mathcal B}^{\alpha,q}_{p,\mathcal F}}
    \lesssim  N^{\frac{3}{2}}2^{N|\alpha|} C_N \|g\chi_{\mathbb R^{n}\setminus S(0,M)}\|_{\dot{\mathcal B}^{\alpha,q}_{p,\mathcal F}}\to 0$$
as $M\to \infty$. Thus,
\begin{equation}\label{3.7}
  \begin{split}
  \langle f,D_{k}(g)\rangle=\lim_{M\to\infty}\langle f,g_{k,M}\rangle.
  \end{split}
  \end{equation}
Since $\{\text{int}\ S(z,2^{-(k+J)})\}_{z\in S(0,M)}$ is an open covering of $S(0,M)$,
there exist finite number of sections $\{S(z_j,2^{-(k+J)})\}_{j=1}^{N_J}$, $z_j\in S(0,M)$,
such that $S(0,M)\subset \bigcup_{j=1}^{N_J} S(z_j,2^{-(k+J)})$. Let
\begin{enumerate}
\item[] $Q_1=S(0,M)\bigcap S(z_1,2^{-(k+J)})$;
\item[] $Q_2=S(0,M)\bigcap S(z_2,2^{-(k+J)})\backslash Q_1$;
\item[] $Q_3=S(0,M)\bigcap S(z_3,2^{-(k+J)})\backslash (Q_1 \cup Q_2)$;
\item[] \hskip1.9cm$\vdots$
\item[] $Q_{N_J}=S(0,M)\bigcap S(z_{N_J}, 2^{-(k+J)})\backslash \bigcup_{j=1}^{N_J-1} Q_j$.
\end{enumerate}
Then $\{Q_j\}_{j=1}^{N_J}$ are disjoint and $\bigcup_{j=1}^{N_J} Q_j=S(0,M)$. Now we write
\begin{align*}
g_{k,M}(x)&=\sum_{j=1}^{N_J}\int_{Q_j} D_k(x,y)g(y)d\mu(y) \\
&=\sum_{j=1}^{N_J}\int_{Q_j} [D_k(x,y)-D_k(x,y_j)]g(y)d\mu(y) \\
&\quad+\sum_{j=1}^{N_J}D_k(x,y_j)\int_{Q_j} g(y)d\mu(y) \\
&:=g^1_{k,M,J}(x)+g^2_{k,M,J}(x),
\end{align*}
where $y_j$ is any point in $Q_j$.
To consider $\|g^1_{k,M,J}\|_{\dot{\mathcal B}^{\alpha,q}_{p,\mathcal F}}$, the second difference smoothness condition (v) in Lemma \ref{id}  will be used.
For simplicity of notations, we denote by
$$F_{k,j}(x,y)= [D_k(x,y)-D_k(x,y_j)]\chi_{{}_{Q_j}}(y).$$
Lemma \ref{id} tells us that
\begin{enumerate}
\item[(a)] $\text{supp}\,F_{k,j}(\cdot,y)\subset S(y, 16(A_0)^2 2^{-k})$\ \ and\ \ $\text{supp}\,F_{k,j}(x,\cdot)\subset S(x, 8A_0 2^{-k})$;
\item[(b)] $\displaystyle \int_{\Bbb R^n} F_{k,j}(x,y)d\mu(x)=0$;
\item[(c)] $\displaystyle |F_{k,j}(x,y)|\le C2^{-J\varepsilon}\frac 1{V_{k}(x)+V_{k}(y)}$;
\item[(d)] $\displaystyle |F_{k,j}(x,y)-F_{k,j}(x',y)|\le C2^{-J\varepsilon}2^{k\varepsilon}({\bar \rho}(x,x'))^\varepsilon\frac 1{V_{k}(x)+V_{k}(y)}$,
\end{enumerate}
where $x'$ satisfies ${\bar \rho}(x,x')\le 32(A_0)^32^{-k}$. Under the above conditions (a)$-$(d), using a similar argument to the proofs of Lemmas \ref{id} and \ref{lem 2.2},
we obtain that for all $k,\ell\in \Bbb Z$ and $x,y\in \Bbb R^n$,
\begin{align}
&\text{supp} (D_\ell F_{k,j})(\cdot,y)\subset S(y, 32(A_0)^3(2^{-\ell} \vee  2^{-k}));  \label{3.15}\\
&\text{supp}  (D_\ell F_{k,j})(x,\cdot)\subset  S(x, 16(A_0)^2(2^{-\ell} \vee  2^{-k}));\label{3.16}\\
&|D_\ell F_{k,j}(x,y)|\le C2^{-J\varepsilon}2^{-|\ell-k|\varepsilon}\frac 1{V_{\ell \wedge k}(x)+V_{\ell \wedge k}(y)}.\label{3.17}
\end{align}
Set
$$F(x,y)=\sum_{j=1}^{N_J} (D_\ell F_{k,j})(x,y).$$
By (\ref{3.16}) and (\ref{3.17}),
\begin{align*}
\int_{\Bbb R^n} |F(x,y)|d\mu(y)
&\le C2^{-J\varepsilon}2^{-|\ell-k|\varepsilon}\sum_{j=1}^{N_J}\int_{Q_j\cap S(x, 16(A_0)^2(2^{-\ell} \vee  2^{-k}))}
                        \frac {d\mu(y)}{V_{\ell \wedge k}(x)+V_{\ell \wedge k}(y)}  \\
&\le C2^{-J\varepsilon}2^{-|\ell-k|\varepsilon}\frac {\mu(S(x, 16(A_0)^2(2^{-\ell} \vee  2^{-k})))}{V_{\ell \wedge k}(x)+V_{\ell \wedge k}(y)} \\
&\le C2^{-J\varepsilon}2^{-|\ell-k|\varepsilon}.
\end{align*}
Similarly, (\ref{3.15}) and (\ref{3.17}) yield
$$\int_{\Bbb R^n} |F(x,y)|d\mu(x)\le C2^{-J\varepsilon}2^{-|\ell-k|\varepsilon}. $$
The above two inequalities imply
$$\|D_\ell(g^1_{k,M,J})\|_{L^p_\mu}\le C2^{-J\varepsilon}2^{-|\ell-k|\varepsilon}\|g\|_{L^p_\mu},$$
and then
\begin{equation}\label{3.20}
\begin{split}
\|g^1_{k,M,J}\|_{\dot{\mathcal B}^{\alpha,q}_{p,\mathcal F}}
&\le C2^{-J\varepsilon}\bigg\{\sum_{\ell\in \Bbb Z} 2^{\ell \alpha q-|\ell-k|\varepsilon q}\bigg\}^{1/q}\|g\|_{L^p_\mu} \\
&\le C2^{-J\varepsilon}2^{k\alpha}\|g\|_{L^p_\mu}\\
&\to 0\qquad\text{as}\ J\to \infty.
\end{split}
\end{equation}
By (\ref{3.7}) and (\ref{3.20}), we have
\begin{equation}\label{3.21}
\begin{split}
    \langle f, D_k(g)\rangle
%     =&\lim_{M\to \infty}\langle f, g_{k,M}\rangle\\
    =&\lim_{M\to \infty}\lim_{J\to \infty}\langle f, g^2_{k,M,J}\rangle\\
    =&\lim_{M\to \infty}\lim_{J\to \infty} \sum_{j=1}^{N_J}D_k(f)(y_j)\int_{Q_j} g(y)d\mu(y),
\end{split}
\end{equation}
where we use Lemma \ref{id} (i) to know that $D_k$ is self-adjoint.
We now write
 \begin{align*}
& \sum_{j=1}^{N_J}D_k(f)(y_j)\int_{Q_j} g(y)d\mu(y) \\
&\quad=\sum_{j=1}^{N_J}\int_{Q_j} D_k(f)(y)g(y)d\mu(y)  \\
&\qquad+\int_{\Bbb R^n} \bigg\{\sum_{j=1}^{N_J}[D_k(f)(y_j)-D_k(f)(y)]\chi_{{}_{Q_j}} \bigg\}g(y)d\mu(y).
\end{align*}
Using the second difference property (v) in Lemma {\ref{id}} again and a similar proof of (\ref{3.20}),
we can show that
$$\big\|[D_k(y_j,\cdot)-D_k(y,\cdot)]\chi_{{}_{Q_j}} \big\|_{\dot{\mathcal B}^{\alpha,q}_{p,\mathcal F}}
\le C2^{-J\varepsilon}2^{k\alpha}V_{k}(y)^{\frac{1}{p}-1}$$
and hence
\begin{align*}
\big|[D_k(f)(y_j)-D_k(f)(y)]\chi_{{}_{Q_j}} \big|&\le \big\|[D_k(y_j,\cdot)-D_k(y,\cdot)]\chi_{{}_{Q_j}} \big\|_{\dot{\mathcal B}^{\alpha,q}_{p,\mathcal F}}\|f\|_{({\dot{\mathcal B}^{\alpha,q}_{p,\mathcal F}})'}\\
&\leq C2^{-J\varepsilon}2^{k\alpha}V_{k}(y)^{\frac{1}{p}-1}\|f\|_{({\dot{\mathcal B}^{\alpha,q}_{p,\mathcal F}})'}.
\end{align*}
The Lebesgue dominated convergence theorem shows that
$$\lim_{J\to \infty} \int_{\Bbb R^n} \bigg\{\sum_{j=1}^{N_J}[D_k(f)(y_j)-D_k(f)(y)]\chi_{{}_{Q_j}} \bigg\}g(y)d\mu(y)=0,$$
which together with (\ref{3.21}) shows
\begin{align*}
\langle f, D_{k}(g) \rangle
&=\lim_{M\to \infty}\lim_{J\to \infty} \sum_{j=1}^{N_J}\int_{Q_j} D_k(f)(y)g(y)d\mu(y) \\
&=\int_{\Bbb R^n} D_k(f)(y)g(y)d\mu(y)  \\
&=\langle D_k(f), g \rangle.
\end{align*}

For the proof of {\bf step 3}, given $g\in\dot{\mathcal B}^{\alpha,q}_{p,\mathcal F}$, let
$$\widetilde g_{k,M}(x)=\int_{S(0,M)}D_{k}^{N}(x,y)D_{k}T_{N}^{-1}(g)(y)d\mu(y),\qquad M>0.$$
Then $\widetilde g_{k,M}\in L^{p}_{\mu}\cap\dot{\mathcal B}^{\alpha,q}_{p,\mathcal F}$.
It follows from Theorem \ref{main1} that
$$\bigg\|g-\sum_{|k|\leq M}\widetilde g_{k,M}\bigg\|_{\dot{\mathcal B}^{\alpha,q}_{p,\mathcal F}}
=\bigg\|g-\sum_{|k|\leq M}D_{k}^{N}D_{k}T_{N}^{-1}(g)\chi_{S(0,M)}\bigg\|_{\dot{\mathcal B}^{\alpha,q}_{p,\mathcal F}}
\to 0\qquad\text{as}\ M\to\infty.$$
Hence, claim \eqref{3.51} is proved, and the proof of Theorem \ref{main2} is completed.
\end{proof}

%%%%%%%%%%%%%%%%%%%%%%%%%%%%%%%%%%%%%%%%%%%%%%%%%%%%%%%%%%%%%
%%%%%%%%%%%%%%%%%%%%%%%%%%%%%%%%%%%%%%%%%%%%
%%%%%%%%%%%%%%%%%%%%%%%%%%%%%%%%%%%%%%%%%%%%%%%%%%%%%%%%%%%%%
\section{Besov spaces associated with sections}\label{besov}

In this section, we study the basic properties of Besov spaces. We first apply the Calder\'on-type reproducing formula for $L^{2}_{\mu}$ to prove that the definition of $\dot{\mathcal B}^{\alpha,q}_{p,\mathcal F}$ is independent of the choice
of approximations to the identity.

\begin{pro}\label{prop 3.2}
Let $|\alpha|<\frac{\varepsilon}{4}$ and $1\leq p,q\leq\infty$.
Suppose that $\{S_k\}_{k\in \Bbb Z}$ and $\{P_k\}_{k\in \Bbb Z}$ are approximations to the identity associated with $\mathcal F$.
Set $D_k=S_k-S_{k-1}$ and $E_k=P_k-P_{k-1}$. Then for $f\in L^2_\mu$,
$$\bigg\{\sum_{k\in \Bbb Z} \Big(2^{k\alpha}\|D_k(f)\|_{L^p_\mu}\Big)^q\bigg\}^{1/q} \approx \bigg\{\sum_{k\in \Bbb Z} \Big(2^{k\alpha}\|E_k(f)\|_{L^p_\mu}\Big)^q\bigg\}^{1/q}.$$
\end{pro}

\begin{proof}
For $f\in L^2_\mu$, we have $f=\sum_{k'\in \Bbb Z} T_N^{-1}E_{k'}^NE_{k'}(f)$ in $L^2_\mu$. Hence
$$D_k(f)
=\sum_{k'\in \Bbb Z} D_kT_N^{-1}E_{k'}^NE_{k'}(f)
=\sum_{k'\in \Bbb Z}\sum_{m=0}^\infty   D_k(R_N)^m E_{k'}^NE_{k'}(f).
$$
Applying the same argument as the proof of \eqref{3.002}, we obtain
$$\bigg\{\sum_{k\in \Bbb Z} \Big(2^{k\alpha}\|D_k(f)\|_{L^p_\mu}\Big)^q\bigg\}^{1/q}
\le C\bigg\{\sum_{k'\in \Bbb Z} \Big(2^{k'\alpha}\|E_{k'}(f)\|_{L^p_\mu}\Big)^q\bigg\}^{1/q}$$
and hence the proof follows.
\end{proof}

It is well known that the space of Schwartz functions is dense in the classical Besov space on $\Bbb R^n$. The following result is one of the main results in this section, which shows that the test function space
$\dot{\mathcal B}^{\alpha,q}_{p,\mathcal F}$ is dense in $\dot B^{\alpha,q}_{p,\mathcal F}$ as well.

\begin{thm}\label{thm 3.4}
Let $|\alpha|<\varepsilon/4$ and $1\le p,q\le \infty$. Then
$$\overline{\dot{\mathcal B}^{\alpha,q}_{p,\mathcal F}}=\dot B^{\alpha,q}_{p,\mathcal F},$$
where $\overline{\dot{\mathcal B}^{\alpha,q}_{p,\mathcal F}}$ denotes the closure of $\dot{\mathcal B}^{\alpha,q}_{p,\mathcal F}$
 with respect to $\|\cdot\|_{\dot {B}^{\alpha,q}_{p,\mathcal F}}.$
\end{thm}
To show the above theorem, we need the following lemma.

\begin{lem}\label{lem 3.3}
Let $\{S_k\}_{k\in \Bbb Z}$ be an approximation to the identity associated with $\mathcal F$
and $D_k=S_k-S_{k-1}$ for $k\in \Bbb Z$. For $|\alpha|<\varepsilon/4$ and $1\le p,q\le \infty$,
both $D_k(\cdot,y)$ and $D_k(x,\cdot)$ are in $\dot{\mathcal B}^{\alpha,q}_{p,\mathcal F}$ for all $x,y\in \Bbb R^n$ and $k\in \Bbb Z$.
\end{lem}

\begin{proof}
Since $D_k(x,\cdot)=D_k(\cdot,x)$ for any fixed $x\in \Bbb R^n$, it suffices to verify the lemma for $D_k(x,\cdot)$.
By Lemma \ref{lem 2.2},
$$\|D_j(D_k(x,\cdot))\|_{L^1_\mu}\le C2^{-|j-k|\varepsilon}$$
and
$$\|D_j(D_k(x,\cdot))\|_{L^\infty_\mu}\le C2^{-|j-k|\varepsilon}\frac 1{V_k(x)}.$$
If $1<p<\infty$, then
$$\|D_j(D_k(x,\cdot))\|_{L^p_\mu}
\le \|D_j(D_k(x,\cdot))\|_{L^1_\mu}^{1/p}\|D_j(D_k(x,\cdot))\|_{L^\infty_\mu}^{1-1/p}
\le C2^{-|j-k|\varepsilon}{V_k(x)^{1/p-1}}.$$
Combining above three estimates yields
\begin{align*}
 \|D_k(x,\cdot)\|_{\dot{\mathcal B}^{\alpha,q}_{p,\mathcal F}}
&=\bigg\{\sum_{j\in \Bbb Z} \Big(2^{j\alpha}\|D_j(D_k(x,\cdot))\|_{L^p_\mu}\Big)^q\bigg\}^{1/q} \\
&\le C\frac 1{V_k(x)^{1-1/p}}\bigg\{\sum_{j\in \Bbb Z} 2^{j\alpha q-|j-k|\varepsilon q}\bigg\}^{1/q} \\
&\le C2^{k\alpha}\frac 1{V_k(x)^{1-1/p}},
\end{align*}
and the proof of the lemma \ref{lem 3.3} is completed.
\end{proof}

\begin{rem}\label{rem 3.1}
The same argument as the proof of Lemma \ref{lem 3.3} shows that if $f\in C^1(\Bbb R^n)$ with compact support and
$$\int_{\Bbb R^n} f(x)d\mu(x)=0,$$
then $f\in \dot{\mathcal B}^{\alpha,q}_{p,\mathcal F}$ for $|\alpha|<\varepsilon/4$ and $1\le p,q\le \infty$.
\end{rem}

We now show Theorem \ref{thm 3.4}.

\vskip 0.2cm\noindent
{\it Proof of Theorem \ref{thm 3.4}.}
To show $\overline{\dot{\mathcal B}^{\alpha,q}_{p,\mathcal F}}\subset \dot B^{\alpha,q}_{p,\mathcal F}$,
let $\{f_m\}_{m\in \Bbb N}$ be a Cauchy sequence in $\dot {\mathcal B}^{\alpha,q}_{p,\mathcal F}$ with respect to the norm $\|\cdot\|_{\dot {B}^{\alpha,q}_{p,\mathcal F}}$.
We will prove that there is an $f\in \dot{B}^{\alpha,q}_{p,\mathcal F}$ such that $f_m$ converges to $f$ in $\dot {B}^{\alpha,q}_{p,\mathcal F}$ as $m\to \infty$.

We first claim that, if $f\in \dot{\mathcal B}^{\alpha,q}_{p,\mathcal F}$,
then $f\in \big(\dot{\mathcal B}^{-\alpha,q'}_{p',\mathcal F}\big)'$ and
$\|f\|_{(\dot {\mathcal B}^{-\alpha,q'}_{p',\mathcal F})'}\le C\|f\|_{\dot{\mathcal B}^{\alpha,q}_{p,\mathcal F}}$.
Given $f\in \dot {\mathcal B}^{\alpha,q}_{p,\mathcal F}$ and
$g\in \dot {\mathcal B}^{-\alpha,q'}_{p',\mathcal F}$,
let $\{S_k\}_{k\in \Bbb Z}$ be an approximation to the identity associated to sections and set $D_k=S_k-S_{k-1}$.
By Calder\'on-type reproducing formula $f=\sum_{k\in \Bbb Z} D_k^ND_kT_N^{-1}(f)$ in $L^2_\mu$ and H\"older's inequality,
\begin{align*}
|\langle f,g \rangle|
%& =\bigg|\int_{\Bbb R^n} \sum_{k\in \Bbb Z} D^N_k D_k T^{-1}_N(f) g d\mu \bigg|\\
&=\bigg|\int_{\Bbb R^n} \sum_{k\in \Bbb Z}  D_k T^{-1}_N(f) D^N_k(g) d\mu \bigg|\\
&\le \sum_{k\in \Bbb Z} \|D_k T^{-1}_N(f)\|_{L^p_\mu} \|D^N_k(g)\|_{L^{p'}_\mu} \\
&\le \bigg\{\sum_{k\in \Bbb Z} 2^{k\alpha q} \|D_k T^{-1}_N(f)\|_{L^p_\mu}^q \bigg\}^{1/q}\bigg\{\sum_{k\in \Bbb Z} 2^{-k\alpha q'}\|D^N_k(g)\|_{L^{p'}_\mu}^{q'}\bigg\}^{1/q'}.
\end{align*}
Since $D^N_k=\sum_{|j|\le N} D_{j+k}$, we have
$$\bigg\{\sum_{k\in \Bbb Z} 2^{-k\alpha q'}\|D^N_k(g)\|_{L^{p'}_\mu}^{q'}\bigg\}^{1/q'}\le CN2^{N|\alpha|}\|g\|_{\dot{\mathcal B}^{-\alpha,q'}_{p',\mathcal F}},$$
and hence
\begin{equation}\label{4.1}
|\langle f,g \rangle|
\le CN2^{N|\alpha|}C_N\|f\|_{\dot{\mathcal B}^{\alpha,q}_{p,\mathcal F}}\|g\|_{\dot{\mathcal B}^{-\alpha,q'}_{p',\mathcal F}}
\end{equation}
 due to \eqref{3.1}.
Thus, the claim follows.

Let $\{f_m\}_{m\in \Bbb N}\subset \dot {\mathcal B}^{\alpha,q}_{p,\mathcal F}$ be a Cauchy sequence with respect to the norm $\|\cdot\|_{\dot {B}^{\alpha,q}_{p,\mathcal F}}$.
The above claim implies that $\{f_m\}_{m\in \Bbb N}$ is also Cauchy with respect to the norm
$\|\cdot\|_{(\dot {\mathcal B}^{-\alpha,q'}_{p',\mathcal F})'}$ and
$\|f_m\|_{\dot {B}^{\alpha,q}_{p,\mathcal F}}\le C$ with $C$ independent of $m$.
Since ${\big(\dot {\mathcal B}^{-\alpha,q'}_{p',\mathcal F}\big)'}$ is a Banach space, there is an
$f\in {\big(\dot {\mathcal B}^{-\alpha,q'}_{p',\mathcal F}\big)'}$ such that $f_m\to f$ in ${\big(\dot {\mathcal B}^{-\alpha,q'}_{p',\mathcal F}\big)'}$ as $m\to \infty$.
It follows from Lemma \ref{lem 3.3} that
$$|D_k(f_m-f)(x)|\le \|D_k(x,\cdot)\|_{\dot {\mathcal B}^{-\alpha,q'}_{p',\mathcal F}}\|f_m-f\|_{(\dot {\mathcal B}^{-\alpha,q'}_{p',\mathcal F})'},$$
which shows
\begin{equation}\label{3.107}
\lim_{m\to \infty} D_k(f_m)(x)=D_k(f)(x).
\end{equation}
By Fatou's lemma and \eqref{3.107},
$$
\|f\|_{\dot {B}^{\alpha,q}_{p,\mathcal F}}\le \liminf_{m\to \infty}\|f_m\|_{\dot {B}^{\alpha,q}_{p,\mathcal F}}\le C,
$$
which shows $f\in \dot {B}^{\alpha,q}_{p,\mathcal F}$.
By the Lebesgue dominated convergence theorem, we obtain that $\{f_m\}$ converges to $f$ in $\dot {B}^{\alpha,q}_{p,\mathcal F}$.

To prove $\dot B^{\alpha,q}_{p,\mathcal F} \subset \overline{\dot{\mathcal B}^{\alpha,q}_{p,\mathcal F}}$,
given $f\in \dot B^{\alpha,q}_{p,\mathcal F}$, the same argument as proof of Theorem \ref{main2} shows
\begin{equation}\label{3.108}
f=\sum_{k\in \Bbb Z} D_kD^N_kT_N^{-1}(f),
\end{equation}
where the series converges in $\dot B^{\alpha,q}_{p,\mathcal F}$. Define $f_{k,M}$ by
$$f_{k,M}(x)=\int_{S(0,M)} D_k(x,y)(D^N_kT^{-1}_N)(f)(y)d\mu(y).$$
Then
$$\bigg\|f-\sum_{|k|\le M} f_{k,M}\bigg\|_{\dot B^{\alpha,q}_{p,\mathcal F}}
\le \bigg\|\sum_{|k|\le M}D_kD^N_kT_N^{-1}(f)-\sum_{|k|\le M} f_{k,M}\bigg\|_{\dot B^{\alpha,q}_{p,\mathcal F}} + \bigg\|\sum_{|k|>M} D_kD^N_kT_N^{-1}(f)\bigg\|_{\dot B^{\alpha,q}_{p,\mathcal F}}.$$
Minkowski's inequality and Lemma \ref{lem 2.3} yield
\begin{align*}
&\bigg\|\sum_{|k|\le M}D_kD^N_kT_N^{-1}(f)-\sum_{|k|\le M} f_{k,M}\bigg\|_{\dot B^{\alpha,q}_{p,\mathcal F}}\\
&\qquad \le \bigg\{\sum_{\ell\in \Bbb Z} \bigg(2^{\ell \alpha}\sum_{|k|\le M}\big\| D_\ell D_k(D^N_kT^{-1}_N(f)\chi_{{}_{\Bbb R^n\backslash S(0,M)}})\big\|_{L^p_\mu}\bigg)^q\bigg\}^{1/q}\\
&\qquad\le C\bigg\{\sum_{\ell\in \Bbb Z} \bigg(\sum_{|k|\le M}2^{(\ell-k) \alpha }2^{-|\ell-k|\varepsilon }2^{k\alpha }
                \big\|D^N_kT^{-1}_N(f)\chi_{{}_{\Bbb R^n\backslash S(0,M)}}\big\|_{L^p_\mu}\bigg)^q\bigg\}^{1/q}.
\end{align*}
By H\"older's inequality,
\begin{align*}
&\bigg\|\sum_{|k|\le M}D_kD^N_kT_N^{-1}(f)-\sum_{|k|\le M} f_{k,M}\bigg\|_{\dot B^{\alpha,q}_{p,\mathcal F}}\\
&\qquad\le C\bigg\{\sum_{\ell\in \Bbb Z} \sum_{|k|\le M} 2^{(\ell-k) \alpha}2^{-|\ell-k|\varepsilon}2^{k\alpha q}
                \|D^N_kT^{-1}_N(f)\chi_{{}_{\Bbb R^n\backslash S(0,M)}}\|_{L^p_\mu}^q\bigg\}^{1/q}.
\end{align*}
Using \eqref{3.1} and \eqref{3.108}, we obtain
$$\bigg\|f-\sum_{|k|\le M} f_{k,M}\bigg\|_{\dot B^{\alpha,q}_{p,\mathcal F}}
  \to 0\qquad\text{as}\ M\to \infty.$$
It follows from Remark \ref{rem 3.1} that $f_{k,M}$ belongs to $\dot {\mathcal B}^{\alpha,q}_{p,\mathcal F}$, so we have $\dot B^{\alpha,q}_{p,\mathcal F} \subset \overline{\dot{\mathcal B}^{\alpha,q}_{p,\mathcal F}}$ and the proof is completed.
%Using \eqref{3.108}, we may further verify that $f$ can be approximated by a finite sum of $f_{k,M}$.
\qed

The following duality argument of $\dot B^{\alpha,q}_{p,\mathcal F}$ is another main result in this section.
\begin{thm}\label{thm 5.1}
Let $|\alpha|<\varepsilon/4$.
\begin{enumerate}
\item[(a)]
  For $1\le p, q\le \infty$ and  each $g\in \dot B^{-\alpha,q'}_{p',\mathcal F}$, the mapping ${\mathcal L}_g: f\mapsto \int_{\Bbb R^n} f(x)g(x)d\mu(x),$
  defined initially on $\dot {\mathcal B}^{\alpha,q}_{p,\mathcal F}$, extends to a bounded linear functional on $\dot B^{\alpha,q}_{p,\mathcal F}$ and satisfies
  $\|{\mathcal L}_g\|\lesssim\|g\|_{\dot {B}^{-\alpha,q'}_{p',\mathcal F}}$.
\item[(b)] Conversely, for $1\le p, q<\infty$, every bounded linear functional $\mathcal L$ on $\dot B^{\alpha,q}_{p,\mathcal F}$ can be realized as ${\mathcal L}={\mathcal L}_g$
  with $g\in \dot B^{-\alpha,q'}_{p',\mathcal F}$ and $\|g\|_{\dot {B}^{-\alpha,q'}_{p',\mathcal F}}\lesssim\|{\mathcal L}\|$.
\end{enumerate}
\end{thm}
To show the above theorem, we need the following

\begin{lem}\label{lem 5.2}
Let $\{S_k\}_{k\in \Bbb Z}$ be an approximation to the identity associated with $\mathcal F$
and $D_k=S_k-S_{k-1}$ for $k\in \Bbb Z$. For $|\alpha|<\varepsilon/4$ and $1\le p,q\le \infty$,
if a sequence of functions $\{g_k\}_{k\in \Bbb Z}$ satisfies
$\big\|\{2^{k\alpha}\|g_k\|_{L^p_\mu}\}_{k\in \Bbb Z}\big\|_{\ell^q}<\infty,$
then $\sum_{k\in \Bbb Z} D_k(g_k)\in \dot B^{\alpha,q}_{p,\mathcal F}$ and
$\big\|\sum_{k\in \Bbb Z} D_k(g_k)\big\|_{\dot {B}^{\alpha,q}_{p,\mathcal F}}
 \lesssim \big\|\{2^{k\alpha}\|g_k\|_{L^p_\mu}\}_{k\in \Bbb Z}\big\|_{\ell^q}.$
\end{lem}

\begin{proof}
For $m_1, m_2\in \Bbb Z$ with $m_1<m_2$, define
$g_{m_1}^{m_2}=\sum_{k=m_1}^{m_2} D_k(g_k)$.
Given $f\in \dot {\mathcal B}^{-\alpha,q'}_{p',\mathcal F}$, H\"older's inequality yields
\begin{align*}
|\langle g_{m_1}^{m_2}, f \rangle|
%&=\bigg|\sum_{k=m_1}^{m_2}\langle D_k(g_k) , f \rangle\bigg| \\
&\le\sum_{k=m_1}^{m_2}|\langle g_k , D_k(f) \rangle| \\
%&\le \sum_{k=m_1}^{m_2}\|g_k\|_{L^p_\mu}\|D_k(f)\|_{L^{p'}_\mu} \\
&\le \bigg\{\sum_{k=m_1}^{m_2} \Big(2^{k\alpha}\|g_k\|_{L^p_\mu}\Big)^q\bigg\}^{1/q}
      \bigg\{\sum_{k=m_1}^{m_2} \Big(2^{-k\alpha}\|D_k(f)\|_{L^{p'}_\mu}\Big)^{q'}\bigg\}^{1/q'}\\
&\le  \bigg\{\sum_{k=m_1}^{m_2} \Big(2^{k\alpha}\|g_k\|_{L^p_\mu}\Big)^q\bigg\}^{1/q}\|f\|_{\dot {\mathcal B}^{-\alpha,q'}_{p',\mathcal F}},
\end{align*}
which shows $ g_{m_1}^{m_2}\in \big(\dot {\mathcal B}^{-\alpha,q'}_{p',\mathcal F}\big)'$ and
$$\|g_{m_1}^{m_2}\|_{(\dot {\mathcal B}^{-\alpha,q'}_{p',\mathcal F})'}\le \bigg\{\sum_{k=m_1}^{m_2} \Big(2^{k\alpha}\|g_k\|_{L^p_\mu}\Big)^q\bigg\}^{1/q}.$$
If we set $g=\sum_{k\in \Bbb Z} D_k(g_k)$, then $g\in \big(\dot {\mathcal B}^{-\alpha,q'}_{p',\mathcal F}\big)'$ as well.  Using Lemma \ref{lem 2.3}
and H\"older's inequality,
we get
\begin{align*}
\sum_{j\in \Bbb Z} \Big(2^{j\alpha}\|D_j(g)\|_{L^p_\mu}\Big)^q
&\le \sum_{j\in \Bbb Z} \bigg(2^{j\alpha}\sum_{k\in \Bbb Z} \|D_jD_k(g_k)\|_{L^p_\mu} \bigg)^q \\
&\lesssim  \sum_{j\in \Bbb Z} \bigg(\sum_{k\in \Bbb Z} 2^{(j-k)\alpha-|j-k|\varepsilon}2^{k\alpha}\|g_k\|_{L^p_\mu} \bigg)^q
%&\le \sum_{j\in \Bbb Z} \bigg(\sum_{k\in \Bbb Z} 2^{(j-k)\alpha-|j-k|\varepsilon}2^{k\alpha %q}\|g_k\|_{L^p_\mu}^q\bigg)
%         \bigg(\sum_{k\in \Bbb Z} 2^{(j-k)\alpha-|j-k|\varepsilon}\bigg)^{q/q'} \\
%&\lesssim  \sum_{j\in \Bbb Z} \sum_{k\in \Bbb Z} 2^{(j-k)\alpha-|j-k|\varepsilon}2^{k\alpha %q}\|g_k\|_{L^p_\mu}^q\\
\lesssim  \sum_{k\in \Bbb Z} 2^{k\alpha q}\|g_k\|_{L^p_\mu}^q,
\end{align*}
which completes the proof.
\end{proof}

Now we return to proving the duality for $\dot B^{\alpha,q}_{p,\mathcal F}$.

\vskip 0.2cm\noindent
{\it Proof of Theorem \ref{thm 5.1}.}
(a) follows from \eqref{4.1} and Theorem \ref{thm 3.4}.
For (b), given a bounded linear functional ${\mathcal L}$ on $\dot B^{\alpha,q}_{p,\mathcal F}$,
by Theorem \ref{thm 3.4} again, ${\mathcal L}$ is also a bounded linear functional on $\dot {\mathcal B}^{\alpha,q}_{p,\mathcal F}$ and
$$|{\mathcal L}(f)|\le \|{\mathcal L}\|\|f\|_{\dot {\mathcal B}^{\alpha,q}_{p,\mathcal F}}\quad\text{for}\ f\in \dot {\mathcal B}^{\alpha,q}_{p,\mathcal F}.$$
Let $\{S_k\}_{k\in \Bbb Z}$ be an approximation to the identity associated with $\mathcal F$
and set $D_k=S_k-S_{k-1}$. Then, for each $f\in \dot {\mathcal B}^{\alpha,q}_{p,\mathcal F}$,
$\{D_k(f)\}_{k\in \Bbb Z}$ is in the sequence space
$$\ell^\alpha_q(L^p_\mu)=\bigg\{\{f_k\}_{k\in \Bbb Z} : \|\{f_k\}_{k\in \Bbb Z}\|_{\ell^\alpha_q(L^p_\mu)}
   :=\bigg(\sum_{k\in \Bbb Z} 2^{k\alpha q}\|f_k\|_{L^p_\mu}^q\bigg)^{1/q}<\infty\bigg\}.$$
Define ${\mathcal L_0}$ on a subset of $\ell^\alpha_q(L^p_\mu)$ by
$${\mathcal L_0}\big(\{D_k(f)\}_{k\in \Bbb Z}\big)={\mathcal L}(f)\qquad\text{for}\ f\in \dot {\mathcal B}^{\alpha,q}_{p,\mathcal F}.$$
Hence,
$$|{\mathcal L_0}\big(\{D_k(f)\}_{k\in \Bbb Z}\big)|
\le \|{\mathcal L}\|\|f\|_{\dot {\mathcal B}^{\alpha,q}_{p,\mathcal F}}
=\|{\mathcal L}\|\|\{D_k(f)\}_{k\in \Bbb Z}\|_{\ell^\alpha_q(L^p_\mu)}.$$
%\begin{align*}
%|{\mathcal L_0}\big(\{D_k(f)\}_{k\in \Bbb Z}\big)|
%&\le \|{\mathcal L}\|\|f\|_{\dot {\mathcal B}^{\alpha,q}_{p,\mathcal F}}\\
%&=\|{\mathcal L}\|\|\{D_k(f)\}_{k\in \Bbb Z}\|_{\ell^\alpha_q(L^p_\mu)}.
%\end{align*}
The Hahn-Banach theorem shows that ${\mathcal L_0}$ can be extended to a functional $\overline {\mathcal L_0}$ on $\ell^\alpha_q(L^p_\mu)$.
Since $(\ell^\alpha_q(L^p_\mu))'=\ell^{-\alpha}_{q'}(L^{p'}_\mu)$ for $1\le p, q<\infty$ (see \cite[page 178]{T1}),
there exists a unique sequence $\{g_k\}_{k\in \Bbb Z}\in \ell^{-\alpha}_{q'}(L^{p'}_\mu)$
such that
$$\overline {\mathcal L_0}(\{f_k\}_{k\in \Bbb Z})=\sum_{k\in \Bbb Z} \langle f_k, g_k \rangle \qquad\text{for all}\ \{f_k\}_{k\in \Bbb Z}\in \ell^\alpha_q(L^p_\mu)$$
and
$$\|\{g_k\}_{k\in \Bbb Z}\|_{\ell^{-\alpha}_{q'}(L^{p'}_\mu)}
\lesssim\|\overline {\mathcal L_0}\|\le \|{\mathcal L}\|.$$
For $f\in \dot {\mathcal B}^{\alpha,q}_{p,\mathcal F}$, we have
$${\mathcal L}(f)
={\mathcal L_0}(\{D_k(f)\}_{k\in \Bbb Z})=\sum_{k\in \Bbb Z} \langle D_k(f), g_k \rangle \\
=\sum_{k\in \Bbb Z} \langle f, D_k(g_k) \rangle =\bigg\langle   f, \sum_{k\in \Bbb Z}D_k(g_k)\bigg\rangle.
$$
Let $g=\sum_{k\in \Bbb Z} D_k(g_k)$.  Lemma \ref{lem 5.2} says that $g\in \dot B^{-\alpha,q'}_{p',\mathcal F}$ and
$$\|g\|_{\dot {B}^{-\alpha,q'}_{p',\mathcal F}}\lesssim\|\{g_k\}_{k\in \Bbb Z}\|_{\ell^{-\alpha}_{q'}(L^{p'}_\mu)}
\lesssim\|{\mathcal L}\|.$$
This completes the proof.
\qed

%%%%%%%%%%%%%%
%%%%%%%%%%%%%%%%%%%%%%%%%%%%%%%%%%%%%%%%%%%%%%
\section{The boundedness on $\dot B^{\alpha,q}_{p,\mathcal F}$}\label{t1}

To prove the boundedness of Monge-Amp\`ere singular integral operator $H$ acting on $\dot B^{\alpha,q}_{p,\mathcal F}$,
the key tool is the almost orthogonality estimate.
A weak version of an almost orthogonality estimate was obtained in \cite[Lemma 9.1]{Li}.
We now show a {\it pointwise} almost orthogonality estimate as follows.
Let $\{E_k\}_{k\in \Bbb Z}$ be an approximation
to the identity associated to sections with regularity exponent $\varepsilon$ and
$$D_k^\#:= D_{-k}= E_{-k} - E_{-k-1}.$$
Denote by $\gamma$ the number satisfying conditions ({\bf D}6) and ({\bf D}7), and by $\epsilon_1$ the constant given in condition ({\bf A}).
The kernel $K(x,y)=\sum_ik_i(x,y)$ of Monge-Amp\`ere singular integral operator $H$
satisfies conditions ({\bf D}1)$-$({\bf D}7), and write
$$H_i(f)(x):=\int_{\Bbb R^n} k_i(x,y)f(y)d\mu(y).$$

\begin{lem}\label{lem 4.2}
For $0<\varepsilon'<\min\{\varepsilon, \gamma \epsilon_1\}$,
\begin{align*}
&|D_k^\# H D_{k'}^\#(x,y)|\\
&\lesssim \frac{2^{-|k-k'|\varepsilon'}}{\mu(S(x,2^{k \vee k'}))+\mu(S(y,2^{k \vee k'}))+\mu\big(S(x,{\bar \rho}(x,y))\big)}\bigg(\frac{2^{k \vee k'}}{2^{k \vee k'}+{\bar \rho}(x,y)}\bigg)^{(\min\{\varepsilon, \gamma \epsilon_1\} - \varepsilon')/2}.
\end{align*}
\end{lem}

\begin{proof}
Obviously
$$|D_k^\#H D_{k'}^\#(x,y)|\le \sum_j|D_k^\#H_j D_{k'}^\#(x,y)|.$$
To estimate $|D_k^\#H_j D_{k'}^\#(x,y)|$, we consider six cases.
As before, we write $V_k(x)=\mu(S(x,2^{-k}))$ for all $x\in \Bbb R^n$ and $k\in \Bbb Z$.

Case 1: $j\le k< k'$. In this case, we use Lemma \ref{id} and conditions ({\bf D}3), ({\bf D}4) to deduce
\begin{align*}
\bigg|\int_{\Bbb R^n}\ k_j(u,v)D_{k'}^\#(v,y)d\mu(v)\bigg|
&=\bigg|\int_{\Bbb R^n}\ k_j(u,v)[D_{k'}^\#(v,y)-D_{k'}^\#(u,y)]d\mu(v)\bigg|\\
&\lesssim  \frac{2^{-|j-k'|\varepsilon}}{V_{-k'}(y)}\int_{\Bbb R^n}\ |k_j(u,v)|d\mu(v) \\
&\le \frac{2^{-|j-k'|\varepsilon}}{V_{-k'}(y)}.
\end{align*}
Note also that the integrand $k_j(u,v)D_{k'}^\#(v,y)$ is zero when ${\bar \rho}(u,y)>9(A_0)^22^{k'}$,
where $A_0$ is the constant satisfying \eqref{eqn:quasitriangleineq}.
For ${\bar \rho}(u,y)\le 9(A_0)^22^{k'}$, the engulfing property of sections
implies $V_{-k'}(u)\approx V_{-k'}(y)$. Therefore
\begin{align*}
\bigg|\int_{\Bbb R^n}\ k_j(u,v)D_{k'}^\#(v,y)d\mu(v)\bigg|
&\lesssim  \frac{2^{-|j-k'|\varepsilon}}{V_{-k'}(y)}\chi_{S(u,9(A_0)^22^{k'})}(y)\\
&\lesssim  \frac{2^{-|j-k'|\varepsilon}}{V_{-k'}(u)+V_{-k'}(y)+\mu\big(S(y,{\bar \rho}(u,y))\big)}
              \bigg(\frac{2^{k'}}{2^{k'}+{\bar \rho}(u,y)}\bigg)^\varepsilon.
\end{align*}
Since $|D_k^\#(x,u)|\lesssim \frac 1{V_{-k}(x)+V_{-k}(u)+\mu(S(x,{\bar \rho}(x,u)))}
              \big(\frac{2^k}{2^k+{\bar \rho}(x,u)}\big)^\varepsilon$, we have
\begin{align*}
|D_k^\#H_j D_{k'}^\#(x,y)|
&\lesssim  \int_{\Bbb R^n}\frac 1{V_{-k}(x)+V_{-k}(u)+\mu\big(S(x,{\bar \rho}(x,u))\big)}
              \bigg(\frac{2^k}{2^k+{\bar \rho}(x,u)}\bigg)^\varepsilon\\
&\qquad\times{\frac {2^{-|j-k'|\varepsilon}}{V_{-k'}(u)+V_{-k'}(y)+\mu\big(S(y,{\bar \rho}(u,y))\big)}
              \bigg(\frac{2^{k'}}{2^{k'}+{\bar \rho}(u,y)}\bigg)^\varepsilon}d\mu(u)\\
&\lesssim \frac{2^{-|j-k'|\varepsilon}}{V_{-k'}(x)+V_{-k'}(y)+\mu\big(S(x,{\bar \rho}(x,y))\big)}
              \bigg(\frac{2^{k'}}{2^{k'}+{\bar \rho}(x,y)}\bigg)^\varepsilon,
\end{align*}
and hence, for $|j-k'|=|j-k|+|k-k'|$,
$$|D_k^\#H_j D_{k'}^\#(x,y)|
  \lesssim \frac{2^{-|j-k|\varepsilon}2^{-|k-k'|\varepsilon}}{V_{-k'}(x)+V_{-k'}(y)+\mu\big(S(x,{\bar \rho}(x,y))\big)}
              \bigg(\frac{2^{k'}}{2^{k'}+{\bar \rho}(x,y)}\bigg)^\varepsilon.$$
Summation over $j\in \Bbb Z$ yields the desired estimate.

Case 2: $j< k'\le k$. Note that the kernel $H_j$ is symmetric for $j< k$. The same argument as in Case 1 gets
$$\bigg|\int_{\Bbb R^n}\ D_k^\#(x,u)k_j(u,v)d\mu(u)\bigg|
\lesssim \frac  {2^{-|j-k|\varepsilon}}{V_{-k}(x)+V_{-k}(v)+\mu\big(S(x,{\bar \rho}(x,v))\big)}
              \Big(\frac{2^k}{2^k+{\bar \rho}(x,v)}\Big)^\varepsilon.$$
Therefore,
$$|D_k^\#H_j D_{k'}^\#(x,y)|
  \lesssim \frac{2^{-|j-k'|\varepsilon} 2^{-|k'-k|\varepsilon}}
  {V_{-k}(x)+V_{-k}(y)+\mu\big(S(x,{\bar \rho}(x,y))\big)}
              \Big(\frac{2^k}{2^k+{\bar \rho}(x,y)}\Big)^\varepsilon$$
and the desired estimate is obtained by taking summation over $j\in \Bbb Z$.

Case 3: $k'\le k< j$. In this case, we use the smoothness condition of $H_j$
and both the cancellation and size conditions of $D_{k'}^\#$ to deduce
\begin{align}
\bigg|\int_{\Bbb R^n}\ k_j(u,v)D_{k'}^\#(v,y)d\mu(v)\bigg|&=\bigg|\int_{S(y,A_02^{3+k'})}\ [k_j(u,v)-k_j(u,y)]D_{k'}^\#(v,y) d\mu(v)\bigg|\label{6.1}\\
&\lesssim   \int_{S(y,A_02^{3+k'})}\ \frac{1}{V_{-j}(u)}|T_j(v)-T_j(y)|^{\gamma}|D_{k'}^\#(v,y)|d\mu(v). \nonumber
\end{align}
We may assume that $S(u, 2^j)\cap S(y, 2^{k'})\not=\varnothing$; otherwise, the integrand is zero.
Hence, by property ({\bf A}) of the sections and $j>k'$,
$$T_j(S(y, 2^{k'}))\subset B\bigg(z,K_1 \Big(\frac{2^{k'}}{2^j} \Big)^{\epsilon_1} \bigg),$$
where $|z|\le K_2$ and $T_j$ is an affine transformation that normalizes $S(u,2^j)$. Therefore,
$$|T_j(v)-T_j(y)|\lesssim 2^{-|j-k'|\epsilon_1},$$
which yields
\begin{align*}
\bigg|\int_{\Bbb R^n}\ k_j(u,v)D_{k'}^\#(v,y)d\mu(v)\bigg|
&\lesssim \frac{2^{-|j-k'|\varepsilon''}}{V_{-j}(u)}\int_{\Bbb R^n}|D_{k'}^\#(v,y)|d\mu(v) \\
&\lesssim \frac{2^{-|j-k'|\varepsilon''}}{V_{-j}(u)}\chi_{S(u,9(A_0)^22^j)}(y),
\end{align*}
where $\varepsilon'' = \frac 12(\min\{\varepsilon, \gamma \epsilon_1\} + \varepsilon')$.
Let $\delta= \frac 12(\min\{\varepsilon, \gamma \epsilon_1\} - \varepsilon')$. We have
\begin{align*}
\bigg|\int_{\Bbb R^n}\ k_j(u,v)D_{k'}^\#(v,y)d\mu(v)\bigg|
&\lesssim  \frac{2^{-|j-k'|\varepsilon''}}{V_{-j}(u)}\Big(\frac{2^j}{2^j+{\bar \rho}(u,y)}\Big)^{\delta}\\
&\lesssim  \frac{2^{-|j-k'|\varepsilon'}}{V_{-k'}(u)+V_{-k'}(y)+\mu\big(S(y,{\bar \rho}(u,y))\big)} \Big(\frac{2^{k'}}{2^{k'}+{\bar \rho}(u,y)}\Big)^{\delta}.
\end{align*}
Arguing as in Case 1, we obtain
$$|D_k^\# H D_{k'}^\#(x,y)|\lesssim \frac{2^{-|k-k'|\varepsilon'}}{V_{-k}(x)+V_{-k}(y)+\mu\big(S(x,{\bar \rho}(x,y))\big)}
              \Big(\frac{2^k}{2^k+{\bar \rho}(x,y)}\Big)^\delta.$$

Case 4: $k<k'\le j$. Similar to Case 3.

Case 5: $k\le j\le k'$. Using the cancellation conditions for $D_k^\#$ and $H_j$ in the second variables, we write
\begin{align*}
&|D_k^\# H_j D_{k'}^\#(x,y)|\\
&=\bigg|\int_{S(x,9(A_0)^22^j)}\int D_k^\#(x,u)[k_j(u,v)-k_j(x,v)][D_{k'}^\#(v,y)-D_{k'}^\#(u,y)]d\mu(u)d\mu(v)\bigg|\\
&\lesssim \frac{2^{-|j-k'|\varepsilon}}{V_{-k'}(y)}\int_{S(x,9(A_0)^22^j)}
\bigg(\int |D_k^\#(x,u)||k_j(u,v)-k_j(x,v)|d\mu(u)\bigg)d\mu(v)
\end{align*}
A similar argument to the estimate for \eqref{6.1} gives
\begin{align*}
|D_k^\# H_j D_{k'}^\#(x,y)|
&\lesssim \frac{2^{-|j-k'|\varepsilon}}{V_{-k'}(y)}\int_{S(x,9(A_0)^22^j)}\frac{2^{-|j-k|\varepsilon'}}{V_{-j}(x)}d\mu(v)\\
&\lesssim \frac{2^{-|j-k'|\varepsilon}}{V_{-k'}(y)}2^{-|j-k|\varepsilon'}\\
&= \frac{2^{-|k-k'|\varepsilon'}}{V_{-k'}(y)}2^{-|j-k'|(\varepsilon-\varepsilon')}.
\end{align*}
Note that the support of $D_k^\# H_j D_{k'}^\#$ forces $\rho(x,y)\lesssim 17(A_0)^{3}2^{k'}$, which implies
$V_{-k'}(y)\approx V_{-k'}(x)$.
Thus,
$$|D_k^\# H_j D_{k'}^\#(x,y)|
  \lesssim \frac{2^{-|k-k'|\varepsilon'}}{V_{-k'}(x)}2^{-|j-k'|(\varepsilon-\varepsilon')}\chi_{S(x,17(A_0)^32^{k'})}(y).$$
Summation over $j\in \Bbb Z$ gives the desire estimate.

Case 6: $k'<j<k$. Similar to Case 5.
\end{proof}

We now are ready to demonstrate Theorem \ref{thm 4.1}.
\begin{proof}[Proof of Theorem \ref{thm 4.1}]
Since $\dot{\mathcal B}^{\alpha,q}_{p,\mathcal F}$ is dense in $\dot B^{\alpha,q}_{p,\mathcal F}$,
it suffices to show Theorem \ref{thm 4.1} for $f\in \dot{\mathcal B}^{\alpha,q}_{p,\mathcal F}$.
Given $f\in \dot{\mathcal B}^{\alpha,q}_{p,\mathcal F}$, we note that $f\in L^2(\Bbb R^n, d\mu)$ and $H$ is bounded on $L^2(\Bbb R^n, d\mu)$.
Applying the Calder\'on-type reproducing formula \eqref{cal4} yields
$$D_k(Hf)(x)
  =D_kH\bigg(\sum_{k'\in \Bbb Z} D^N_{k'}D_{k'}T_N^{-1}(f)\bigg)(x)
  =\sum_{k'\in \Bbb Z} D_kH D^N_{k'}D_{k'}T_N^{-1}(f)(x).$$
By Lemma \ref{lem 4.2} and Minkowski's inequality, we have
$$\|D_k(Hf)\|_{L^p_\mu}\lesssim \sum_{k'\in \Bbb Z} 2^{-|k-k'|\varepsilon'}\|D_{k'}T_N^{-1}(f)\|_{L^p_\mu}.$$
Hence, H\"oder's inequality gives
\begin{align*}
\|H(f)\|_{\dot {B}^{\alpha,q}_{p,\mathcal F}}
%&=\bigg\{\sum_{k\in \Bbb Z} \Big(2^{k\alpha}\|D_k(Hf)\|_{L^p_\mu}\Big)^q\bigg\}^{1/q} \\
&\lesssim \bigg\{\sum_{k\in \Bbb Z} \Big(2^{k\alpha}\sum_{k'\in \Bbb Z} 2^{-|k-k'|\varepsilon'}\|D_{k'}T_N^{-1}(f)\|_{L^p_\mu}\Big)^q\bigg\}^{1/q} \\
%&\le C\bigg\{\sum_{k\in \Bbb Z} \Big(\sum_{k'\in \Bbb Z} 2^{-|k-k'|\varepsilon'+(k-k')\alpha}2^{k'\alpha}\|D^N_{k'}T_N^{-1}(f)\|_{L^p_\mu}\Big)^q\bigg\}^{1/q}\\
&\lesssim  \bigg\{\sum_{k'\in \Bbb Z} \Big(2^{k'\alpha}\|D_{k'}T_N^{-1}(f)\|_{L^p_\mu}\Big)^q\bigg\}^{1/q}.
\end{align*}
By \eqref{3.1}, $\|H(f)\|_{\dot {B}^{\alpha,q}_{p,\mathcal F}}\lesssim C_N\|f\|_{\dot {B}^{\alpha,q}_{p,\mathcal F}}.$
\end{proof}

%%%%%%%%%%%%%%%%%%%%%%%%%%%%%%%%%%%%%%%%%%%%%%
%%%%%%%%%%%%%%
%%%%%%%%%%%%%%%%%%%%%%%%%%%%%%%%%%%%%%%%%%%%%%
%%%%%%%%%%%%%%%%%%%%%%%%%%%%%%%%%%%%%%%%%%%%%%%%
%%%%%%%%%%%%%%%%%%%%%%%%%%%%
%%%%%%%%%%%%%%%%%%%%%%%%%%%%%%%%%%%%%%%%%%%%%%%%

\vskip 0.5cm

\flushleft
{\small Yongshen Han\\
Department of Mathematics\\
Auburn University\\
Auburn, Alabama 36849-5310\\
U.S.A.\\
Email: hanyong@auburn.edu}
\vskip 0.75cm

\flushleft
{\small Ming-Yi Lee \& Chin-Cheng Lin\\
Department of Mathematics\\
National Central University\\
Chung-Li, Taiwan 320\\
Republic of China\\
Email: mylee@math.ncu.edu.tw; clin@math.ncu.edu.tw}


\begin{thebibliography}{99}

\bibitem{C1} L. A. Caffarelli,
\emph{Some regularity properties of solutions of Monge-Amp\`ere equation},
Comm. Pure Appl. Math. {\bf XLIV} (1991), 965--969.

\bibitem{C2} L. A. Caffarelli,
\emph{Boundary regularity of maps with convex potentials},
Comm. Pure Appl. Math. {\bf XLV} (1992), 1141--1151.

\bibitem{CG1} L. A. Caffarelli and C. E. Guti\'errez,
\emph{Real analysis related to the Monge-Amp\`ere equation},
Trans. Amer. Math. Soc. {\bf 348} (1996), 1075--1092.

\bibitem{CG2} L. A. Caffarelli and C. E. Guti\'errez,
\emph{Properties of the solutions of the linearized Monge-Amp\`ere equation},
Amer. J. Math. {\bf 119} (1997), 423--465.

\bibitem{CG3} L. A. Caffarelli and C. E. Guti\'errez,
\emph{Singular integrals related to the Monge-Amp\`ere equation},
        Wavelet Theory and Harmonic Analysis in Applied Sciences
        (Buenos Aires, 1995), 3--13,
        C. A. D'Atellis and E. M. Fernandez-Berdaguer, Eds.,
        Appl. Numer. Harmon. Anal.,
        Birkh\"auser Boston, Boston, MA, 1997.

\bibitem{Ca} A. P. Calder\'on,
\emph{Intermediate spaces and interpolation, the complex method},
Studia Math. {\bf 24} (1964), 113--190.

\bibitem{CW1} R. R. Coifman and G. Weiss,
\emph{Analyse Harmonique Non-Commutative sur Certains Espaces Homogenes},
Lecture Notes in Math. {\bf 242},
Springer-Verlag, Berlin and New York, 1971.

\bibitem{CW2} R. R. Coifman and G. Weiss,
\emph{Extensions of Hardy spaces and their use in analysis},
Bull. Amer. Math. Soc. {\bf 83} (1977), 569--645.

\bibitem{DJS} G. David, J.-L. Journ\'e, and S. Semmes,
\emph{Op\'erateurs de Calder\'on-Zygmund, fonctions para-accr\'etives et interpolation},
Rev. Mat. Iberoamericana \textbf{1} (1985), 1--56.

\bibitem{DL} Y. Ding and C.-C. Lin,
\emph{Hardy spaces associated to the sections},
T\^ohoku Math. J. {\bf 57} (2005), 147--170.

\bibitem{FJW} M. Frazier, B. Jawerth, and G. Weiss,
\emph{Littlewood-Paley Theory and the Study of Function Spaces},
CBMS Regional Conference Series {\textbf 79}, A.M.S., Providence, RI, 1991.

\bibitem{H1} Y. Han,
\emph{Calder\'on-type reproducing formula and the $Tb$ theorem},
Rev. Mat. Iberoamericana {\bf 10} (1994), 51--91.

\bibitem{H2} Y. Han,
\emph{Plancherel-P\^olya type inequality on spaces of homogeneous type and its applications},
Proc. Amer. Math. Soc. {\bf 126} (1998), 3315--3327.

\bibitem{HL}  Y. Han and C.-C. Lin,
\emph{Embedding theorem on spaces of homogeneous type},
J. Fourier Anal. Appl. {\bf 8} (2002), 291--307.

\bibitem{HMY1} Y. Han, D. Muller, and D. Yang,
\emph{Littlewood-Paley characterizations for Hardy spaces on spaces of homogeneous type},
Math. Nachr. {\bf 279} (2006), 1505--1537.

\bibitem{HMY2} Y. Han, D. Muller, and D. Yang,
\emph{A theory of Besov and Triebel-Lizorkin spaces on metric measure spaces modeled on Carnot-Caratheodory spaces},
Abstr. Appl. Anal. vol. {\bf 2008}, Art. ID 893409, 250 pp, 2008.

\bibitem{HS} Y. Han and E. Sawyer ,
\emph{Littlewood-Paley theorem on space of homogeneous type and classical function
spaces}, Mem. Amer. Math. Soc. {\bf 110} (1994), No. 530, 1--126.

\bibitem{In} A. Incognito,
\emph{Weak-type $(1,1)$ inequality for the Monge-Amp\`ere SIO's},
J. Fourier Anal. Appl. \textbf{7} (2001), 41--48.

\bibitem{Le} M.-Y. Lee,
\emph{The boundedness of Monge-Amp\`ere singular integral operators},
J. Fourier Anal. Appl. \textbf{18} (2012), 211--222.

\bibitem{Li} C.-C. Lin,
\emph{Boundedness of Monge-Amp\`ere singular integral operators acting on Hardy spaces and their duals},
Trans. Amer. Math. Soc. \textbf{368} (2016), 3075--3104.

\bibitem{MS1} R. A. Mac\'ias and C. Segovia,
\emph{Lipschitz functions on spaces of homogeneous type},
Adv. in Math. \textbf{33} (1979), 257--270.

\bibitem{MS2} R. A. Mac\'ias and C. Segovia,
\emph{A decomposition into atoms of distributions on spaces of homogeneous type},
Adv. in Math. \textbf{33} (1979), 271--309.

\bibitem{NS1} A. Nagel and E. M. Stein,
\emph{On the product theory of singular integrals},
Rev. Mat. Iberoamericana {\bf 20} (2004), 531--561.

\bibitem{NS2} A. Nagel and E. M. Stein,
\emph{The $\bar{\partial}_b$-complex on decoupled boundarise in $\mathbb{C}^n$},
Ann. of Math. (2) {\bf 164} (2006), 649--713.

\bibitem{P} J. Peetre,
\emph{New Thoughts on Besov Spaces},
Duke Univ. Math. Series, No. 1, Durham, N.C., 1976.

\bibitem{T1} H. Triebel,
\emph{Theory of Function Spaces},
Birkh\"auser Verlag, Basel, 1983.

\bibitem{T2} H. Triebel,
\emph{Theory of Function Spaces, II},
Birkh\"auser Verlag, Basel, 1992.

\end{thebibliography}
\end{document}